\documentclass{amsart}
\usepackage[utf8]{inputenc}
\usepackage{tikz}
\usetikzlibrary{calc, arrows.meta}
\usepackage{amsmath}
\usepackage{amsthm}
\usepackage{amsbsy}
\usepackage{amssymb}
\usepackage{amsfonts}
\usepackage{graphicx}
\usepackage{hyperref}
\hypersetup{colorlinks=true, citecolor = green!50!black}
\usepackage[capitalise]{cleveref}

\newtheorem{thm}{Theorem}[section]
\Crefname{thm}{Theorem}{Theorems}
\newtheorem*{thm*}{Theorem}
\Crefname{thm}{Theorem}{Theorems}
\newtheorem{lem}[thm]{Lemma}
\Crefname{lem}{Lemma}{Lemmas}

\Crefname{prop}{Proposition}{Propositions}
\newtheorem{cor}[thm]{Corollary}
\Crefname{cor}{Corollary}{Corollaries} 

\Crefname{que}{Question}{Questions}

\Crefname{con}{Conjecture}{Conjectures}
\newtheorem{clm}[thm]{Claim}
\Crefname{clm}{Claim}{Claims}

\Crefname{goal}{Goal}{Goals}

\theoremstyle{definition}
\newtheorem{defn}[thm]{Definition}
\Crefname{defn}{Definition}{Definitions}
\newtheorem{ex}[thm]{Example}
\Crefname{ex}{Example}{Examples}
\newtheorem{rem}[thm]{Remark}
\Crefname{rem}{Remark}{Remarks}
\newtheorem{obs}[thm]{Observation}
\Crefname{obs}{Observation}{Observations}

\newcommand{\cyccomp}{\mathcal{Z}}
\newcommand{\cycset}{\mathcal{S}}
\newcommand{\restcyc}{\mathbin{\upharpoonright}}

\newcommand{\flowerradius}{2cm}
\newcommand{\addthickness}{1pt}
\newcommand{\orientationradius}{0.2*\flowerradius}
\newcommand{\basicflower}{\draw [thick] (0,0) ellipse [radius=\flowerradius];}
\newcommand{\orientation}{\draw [thick,->,gray] (20:\orientationradius) arc [start angle=20,end angle=-240,radius=\orientationradius];}
\newcommand{\newnode}[3][]{\node [fill,circle,inner sep=1.3pt,label=#3:#2,#1] at (#3:\flowerradius) {};}
\newcommand{\labelofsep}[5][]{
	\draw [#1] ($(#2:\flowerradius) +(#2-90:\addthickness*#5)$)-- (.5*#2+.5*#3:1.4*\addthickness*#5) -- ($(#3:\flowerradius)+(#3+90:\addthickness*#5)$);
	\draw [->,#1] (.5*#2+.5*#3:1.4*\addthickness*#5) -- (.5*#2+.5*#3:0.15*\flowerradius);
	\path [#1] (.5*#2+.5*#3:.6*\flowerradius) node [align=center]{#4};}
\newcommand{\labelofarcbelow}[5][black]{
	\draw [thick,#1] (#2:\flowerradius-#4*\addthickness) arc [start angle=#2, end angle=#3,radius=\flowerradius-#4*\addthickness];
	\node [label=180+0.5*#2+0.5*#3:\textcolor{#1}{#5}] at (0.5*#2+0.5*#3:\flowerradius) {};}
\newcommand{\labelofarcabove}[5][black]{%
	\draw [thick,#1] (#2:\flowerradius+#4*\addthickness) arc [start angle=#2, end angle=#3,radius=\flowerradius+#4*\addthickness];
	\node [label=.5*#2+.5*#3:\textcolor{#1}{#5}] at (0.5*#2+0.5*#3:\flowerradius) {};}
\colorlet{picturegreen}{green!60!black}
\colorlet{pictureorange}{orange!70!brown}
\colorlet{picturepurple}{red!55!blue}

\begin{document}
\title{Flowers in graph-like spaces}
\author{Ann-Kathrin Elm}
\address{Research conducted at Universität Hamburg, Department of Mathematics, Hamburg, Germany}
\curraddr{Ann-Kathrin Elm: Universität Heidelberg, Department of Computer Science, Heidelberg, Germany}
\email{elm@informatik.uni-heidelberg.de, hendrikheine@gmail.com}
\author{Hendrik Heine}
\begin{abstract}
    One perspective on tree decompositions is that they display (low-order) separations of the underlying graph or matroid.
    The separations displayed by a tree decomposition are necessarily nested.
    In 2013, Clark and Whittle proved the existence of tree decompositions with flowers added in which, up to a natural equivalence, all low-order separations are displayed.
    An important step in that proof is to show that flowers can be extended to maximal flowers.
    In this paper we generalise the notion of a flower to pseudoflowers in graph-like spaces and show for our generalisation, flowers can be extended to maximal flowers.
\end{abstract}
\maketitle

\section{Introduction}

Tree decompositions play an important role in the proof of the Graph Minor Theorem by Robertson and Seymour and have since become a central tool of structural graph theory.
Tangles are a way of fuzzily encoding highly connected regions of a graph $G$:
They are objects that orient every separation of order less than $k$ (this $k$ is then the \emph{order} of the tangle) towards the side where the highly connected region of the graph can be found.
In a tree decomposition of a graph $G$, every edge of the decomposition graph $T$ induces a separation of $G$.
So every $k$-tangle orients every edge of $T$ that induces a separation of $G$ of order less than $k$ and thus points towards a subtree of $T$ whose edges induce separations of $G$ of order at least $k$.
In this way, every tangle induces a subtree of $T$.

A tree of tangles of $G$ is a tree decomposition such that for any two distinct tangles of the same order, the induced subtrees of the decomposition tree are disjoint.
Additionally all separations induced by edges of the decomposition tree have to distinguish two tangles, that is, the two tangles orient the separation differently.
The subtrees of the decomposition tree $T$ induced by the tangles are naturally arranged in a tree-like fashion by $T$, and usually a tree of tangles is seen as a tool to exhibit the resulting tree-like structure of the tangles.

From another viewpoint, a tree of tangles also displays the tree-like structure of the separations of $G$ it induces.
As any two separations induced by a tree of tangles cannot cross, the amount of separations that can be induced by one tree of tangles is very limited.
One way around this restriction is to declare two separations as equivalent if they distinguish the same tangles, and to consider a separation as represented by a tree decomposition if some equivalent separation is induced by the tree decomposition.
This approach is justified by the fact, implied by tangle-tree-duality theorems, that for any two equivalent separations $r$ and $s$ of which one is induced by a tree of tangles there is a sequence $r = p_0, \ldots, p_n = s$ of separations such that the ``difference'' between to consecutive separations is less than a separation that is small in the sense that it does not distinguish tangles.

Even when allowing for this equivalence relation, it is in general not possible to represent all separations in one tree decomposition, even in the special case where the tangles to be distinguished all have the same order $k$ and all contain the same tangle of order $k-1$.
The problem here is that there might be separations $r$ and $s$ that \emph{properly cross}, that is, there are tangles $T_1$, $T_2$, $T_3$ and $T_4$ to be distinguished such that both $T_1$ and $T_2$ orient $r$ the same way, $T_1$ and $T_3$ orient $s$ the same way, $r$ distinguishes $T_1$ from both $T_3$ and $T_4$ and $s$ distinguished $T_1$ from both $T_2$ and $T_4$.
In this case, any separations $r'$ and $s'$ equivalent to $r$ and $s$ respectively also have to cross, and thus cannot be induced by a common tree decomposition.

The concepts of tangles, tree decompositions and trees of tangles in matroid theory are closely related to the corresponding concepts in graph theory.
In the context of finite matroids, the problem of representing separations that properly cross has been solved by adding flowers to the vertices of the decomposition tree.
Flowers are a way of encoding many separations simultaneously that mostly pairwise cross.
Flowers and tree decompositions with added flowers were introduced for $3$-tangles in \cite{3flowers} and then generalised to $k$-tangles in \cite{AikinOxley} and \cite{ClarkWhittle}.
\cite[Theorem 7.1]{ClarkWhittle} implies that for a set $\mathcal{T}$ of tangles of the same order $k$ that all contain the same tangle of order $k-1$, there is a tree of tangles with flowers such that the flowers conform with the tree decomposition and such that every separations distinguishing elements of $\mathcal{T}$ is equivalent to either a separation induced by the tree decomposition or to a separation displayed by one of the flowers.
A central part of this theorem is the use of flowers that distinguish a maximal set of elements of $\mathcal{T}$, and the fact that under certain circumstances for a flower $\Phi$, a separation that is only equivalent to separations that cross separations displayed by $\Phi$ can be added to $\Phi$.

In this paper, as a first step towards a version of the theorem by Clark and Whittle for graph-like spaces (a generalisation of infinite graphs), maximal flowers in graph-like spaces are investigated.
We find a definition of pseudoflower that extends the notion of flower from \cite{ClarkWhittle} while being tailored to taking care of several technical difficulties. A similar extension of flowers for infinite matroids can also be found in \cite{elm2023pseudoflowers}.
We then show that every pseudoflower that is not too small can be extended to another pseudoflower that is maximal in the sense that further subdividing one of its petals does not yield another pseudoflower (\cref{leqmaxexiststangle}).
Finally we show that under certain assumptions on $\mathcal{T}$, these maximal pseudoflowers also distinguish maximally many elements of $\mathcal{T}$ (\cref{maximalflowerwrttangles}).

\section{Tools and terminology}

\subsection{Separations, universes and profiles}

The definition of a graph-like space is rather involved and for the results of this paper other than examples we only use a few properties of the separations of a graph-like space.
For this reason, we will mostly work within the framework of a universe of vertex separations.

A \emph{universe of vertex separations} on ground set $V$ (thought of as vertex set) is a tuple $\mathcal{U}=(U,\wedge, \vee, ^*)$ where $U$ consists of pairs $(A,B)$ with $A\cup B=V$ and is closed under $\wedge$, $\vee$ and $^*$; $\wedge: U\times U \rightarrow U$ is defined via $(A,B)\wedge (C,D)=(A\cap C,B\cup D)$, $\vee: U\times U \rightarrow U$ is defined via $(A,B)\vee (C,D)=(A\cup C, B\cap D)$ and $^*:U\rightarrow U$ is defined via $(A,B)^*=(B,A)$.
The elements of $U$ are called \emph{separations}.

The name ``universe of vertex separations'' derives from a more general definition of universe, in which $U$ is just some set and there need not be a ground set $V$, as defined for example in \cite{ASS}.
Universes of vertex separations are studied in depth under the name of separation systems implemented by set separations in \cite{BowlerKneip}.
The main example for universes of vertex separations arise from graphs and a generalisation of graphs called graph-like spaces.

The separation $(B,A)$ is the \emph{inverse} of the separation $(A,B)$, and $(A,B)$ and $(B,A)$ are the \emph{orientations} of $(A,B)$.
Just as with graphs, the relation $\leq$ on $U$ where $(A,B)\leq (C,D)$ if $A\subseteq C$ and $D\subseteq B$ is a partial order in which $(A,B)\wedge (C,D)$ is the infimum of $(A,B)$ and $(C,D)$ and $(A,B)\vee (C,D)$ is the supremum of $(A,B)$ and $(C,D)$.
Two separations $(A,B)$ and $(C,D)$ that have orientations that are comparable by $\leq$ are called \emph{nested}, and if two separations are not nested then they \emph{cross}.
A \emph{corner} of two crossing separations $(A,B)$ and $(C,D)$ is the supremum or infimum of an orientation of $(A,B)$ and an orientation of $(C,D)$, so two crossing separations have (up to) $8$ corners.
The \emph{order} of a separation $(A,B)$ is the size of $A\cap B$ if that is finite and $\infty$ otherwise.
Just as with graphs, the order function is symmetric and submodular.
A $\leq k$-separation is a separation of order at most $k$, and a $<k$-separation is a separation of order less than $k$.

We also need the notion of a separation system.
The more general definition of abstract separation systems, together with related definitions and several easy facts about them can be found in \cite{ASS}.
We only need the separation systems $S_k$ that consist, for a universe $(U, \wedge ,\vee, ^*)$ of vertex separations and some $k \in \mathbb{N}$, of all separations of $U$ that have order less than $k$.
The partial order of $U$ induces a partial order of $S_k$.
In this context, for two elements $r$ and $s$ of $S_k$, the infimum $r \wedge s$ is always taken in the surrounding universe, not in the partial order of $S_k$.
Thus the infimum always exists, but is not always contained in $S_k$.
The same holds for $r \vee s$.

In order to obtain our results we want the universe of vertex separations to behave well with respect to limits of chains of separations that all have the same order.
In this case, we ask that if $k$ is a non-negative integer and $(A_i,B_i)_{i\in I}$ is a chain (with respect to $\leq$) of separations of order at most $k$, then the supremum of $(A_i,B_i)_{i\in I}$ in $U$ exists, has at most order $k$ and is of the form $(A\cup X, B)$ for $A=\bigcup_{i\in I} A_i$, $B=\bigcap_{i\in I}B_i$ and some set $X\subseteq V$.
We call a universe of vertex separations whose order function has this property \emph{limit-closed}.
Throughout this paper we fix a limit-closed universe of vertex separations $\mathcal{U}$ on vertex set $V$.
Note that a sub-universe of $\mathcal{U}(V)$ whose induced order function is still limit-closed need not be a limit-closed universe of vertex separations.

For a given graph, the universe of separations of that graph clearly is a limit-closed universe of vertex separations.
We will now show that the separations of a graph-like space form a limit-closed universe of vertex separations.
As this and the example in \cref{sec:example} are the only parts of the paper where we actually work with graph-like spaces, we do not give the (slightly involved) definition here but refer the interested reader to \cite{BCC:graphic_matroids} and \cite{GraphlikeContinuaandMenger}.
In a graph-like space, let a pair $(A,B)$ be a separation of that graph-like space if $A \cup B$ is the vertex set of the graph-like space and every arc meeting both $A$ and $B$ contains a vertex of $A \cap B$.
In light of the use of pseudo-arcs in \cite{BCC:graphic_matroids}, one might want to allow for different definitions of separations in graph-like spaces, for example that additionally every pseudo-arc meeting both $A$ and $B$ has to meet $A \cap B$.
Such other possible definitions can be captured for example by the notion of separation of a path space, to be found in \cite{heine2020path}.
The following proof that the universe of separations of a graph-like space also is a limit-closed universe of vertex separations also works for path spaces.

\begin{lem}\label{glsarelcd}
	Let $k\in\mathbb{N}$, let $G$ be a graph-like space and let $(A_i,B_i)_{i \in I}$ be a chain of separations of order $\leq k$.
	Let $A = \bigcup_{i \in I} A_i$ and $B = \bigcap_{i \in I} B_i$. Then there exists $X \subseteq V(\mathcal{P})$ such that $(A \cup X, B)$ is a separation of order at most $k$.    
\end{lem}
\begin{proof}
Let $((A_i,B_i))_{i \in I}$ be an increasing chain of separations of order at most $k\in \mathbb{N}$.
Let $A$ be the union of all $A_i$ and $B$ the intersection of all $B_i$.
If there are $k+1$ disjoint arcs $P_1, \ldots, P_{k+1}$ that meet both $A$ and $B$, then there is some sufficiently large $i$ such that all $P_j$ meet $A_i$.
As all $P_j$ also meet $B_i$, this implies that all $P_j$ meet $A_i \cap B_i$, contradicting the fact that $(A_i,B_i)$ has order at most $k$.
So let $P_1, \ldots, P_l$ be a maximum set of disjoint arcs meeting both $A$ and $B$.
Without loss of generality assume that the first vertex of each $P_j$ is contained in $A$ while its last vertex is contained in $B$.
For each $j$ let $x_j$ be the supremum of $P_j \cap A$ in $P_j$, and let $X$ be the set of all $x_j$.
By maximality of the family of the $P_j$, every vertex in $P_j x_j$ is contained in $A + x_j$, every vertex in $x P_j$ is contained in $B + x_j$, and every vertex of $A \cap B$ has to be contained in some $P_j$ and to be equal to $x_j$.
In particular $A \cap B \subseteq X$ and $|(A \cap B) \cup X| = |X| = l \leq k$.

Assume for a contradiction that there is an arc $Q$ meeting both $A$ and $B$ but not $X$.
Without loss of generality, the first vertex of $Q$ is in $A$ and the last vertex of $Q$ is in $B$.
For each $j$ such that $Q$ meets $P_jx_j$ let $q_j$ be the supremum in $Q$ of $Q \cap P_jx_j$, and let $q$ be the maximum of all $q_j$ (or the first vertex of $Q$ if $Q$ meets no $P_j x_j$).
Similarly, for each $j$ such that $qQ$ meets $x_jP_j$ let $r_j$ be the infimum in $qQ$ of $qQ \cap x_jP_j$ and let $r$ be the minimum of all $r_j$ (or the last vertex of $Q$ if $qQ$ meets no $x_jP_j$).
Then $qQr$ has its first vertex in $A$, its last vertex in $B$ and is internally disjoint from the $P_j$.
By maximality of the family $(P_j)$, either all inner vertices of $qQr$ are contained in $B \setminus A$, or $qQr$ has an inner vertex and all inner vertices are contained in $A \setminus B$.
Consider the case that all inner vertices of $qQr$ are contained in $B \setminus A$, and let $i$ be an index such that $q \in A_i$.
Then for all $i' \geq i$, $q \in A_{i'}$ and $r \in B_{i'}$ while all vertices other than $q$ are contained in $B_{i'} \setminus A_{i'}$.
As $qQr$ has to meet all $A_{i'} \cap B_{i'}$ with $i' \geq i$, $q$ is contained in all $B_{i'}$ and thus in $A \cap B$, a contradiction.
Next consider the case that $qQr$ has an inner vertex $y$ and all inner vertices are contained in $A \setminus B$.
Then $yQr$ starts in $A$ and ends in $B$ and, by maximality of the family $(P_j)$, $r$ is contained in some $P_j$.
As $Q$ avoids $X$, this last vertex comes later than $x_j$ in $P_j$.
So by the maximality of the family $(P_j)$, $P_jq$ without its last vertex does not contain a subarc meeting both $A$ and $B$, hence $x_j \in A \setminus B$ and $r$ is the successor of $x_j$ in $P_j$.
But then $x_j \in A_i \setminus B_i$ for some index $i$, and $r \in B_i$.
Thus the edge $x_j r$ has one end vertex in $A_i \cap B_i$, implying $r \in A_i \subseteq A$ and thus $r \in A \cap B$, contradicting the assumption that $Q$ be disjoint from $X$.
So $Q$ does not exist.
\end{proof}

In the situation of \cref{glsarelcd}, for any two separations of the form $(A\cup X,B)$ and $(A\cup Y,B)$, the infimum $(A\cup (X\cap Y),B)$ is also a separation of the graph-like space.
Thus there is a smallest $X$ for which $(A \cup X, B)$ is a separation, and this $X$ satisfies that $(A \cup X, B)$ is the supremum of the chain $(A_i,B_i)_{i\in I}$ in the universe of separations of the graph-like space.
In the special case of the separations of a graph, the set $X$ is always empty.
In graph-like spaces, that is not necessarily the case.

\subsection{Cyclic orders}\label{sec:cycord}

\begin{defn}\cite{Novak84}
	A set of triples $Z$ in $S\times S\times S$ is a \emph{cyclic order} of the set $S$ if it has the following four properties:
	\begin{itemize}
		\item (cyclic) $\forall a,b,c\in S:(a,b,c)\in Z\Rightarrow (b,c,a)\in Z$
		\item (antisymmetric) $\forall a,b,c\in S:(a,b,c)\in Z\Rightarrow (c,b,a)\notin Z$
		\item (linear) $\forall a,b,c\in S \text{ pairwise distinct}:(a,b,c)\notin Z\Rightarrow (c,b,a)\in Z$
		\item (transitive) $\forall a,b,c,d\in S:(a,b,c)\in Z\wedge (a,c,d)\in Z\Rightarrow (a,b,d)\in Z$.
	\end{itemize}
\end{defn}
So formally a cyclic order is a different set than its underlying ground set, but in this paper (except in the appendix) this distinction is not made.
Note that what is called a cyclic order here is sometimes (also in \cite{Novak84}) called a linear cyclic order or a complete cyclic order, with a cyclic order not necessarily being linear.
The distinction is not made here as all cyclic orders under consideration are linear.

For two elements $a$ and $b$ of $S$, the set of elements $c\in S$ satisfying $(a,c,b)\in Z$ is denoted by $\mathopen]a,b\mathclose[$.
The sets $\mathopen]a,b\mathclose[+a$, $\mathopen]a,b\mathclose[+b$ and $\mathopen]a,b\mathclose[+a+b$ are denoted by $[a,b\mathclose[$, $\mathopen]a,b]$ and $[a,b]$ respectively. When necessary to resolve ambiguities, we may add the cyclic order in which these are taken as a subscript. We call a subset $I$ of $S$ is an \emph{interval} if and only if for all $s,t\in I$ either $[s,t]\subseteq I$ or $[t,s]\subseteq I$.
Clearly, subsets of $S$ of the form $\mathopen]a,b\mathclose[$, $\mathopen]a,b]$, $[a,b\mathclose[$ or $[a,b]$ are all intervals of $Z$.
We call an interval \emph{non-trivial} if it is neither the empty set nor all of $S$.
Similarly to linear orders, an element $s\neq a$ is the \emph{successor} of $a$ in $S$ if $b\notin [a,s]$ for all for all other elements $b$.
Also, an element $p\neq a$ is the \emph{predecessor} of $a$ if $b\notin [p,a]$ for all other elements $b$.
A \emph{neighbor} of $a$ is an element that is a predecessor or successor of $a$.

\begin{rem}\cite[Lemma 1.4]{Novak82}
	The definition of a cyclic order implies that if $Z$ is a cyclic order on a set $S$ and $(s,s',t)$ is an element of $Z$ then $s$, $s'$ and $t$ are pairwise distinct.
\end{rem}

We will use cyclically ordered sets of a special kind as the index set of our $k$-pseudoflowers:

\begin{defn}\label{def:completionofcyclicorder}
Let $I$ be a cyclically ordered set with at least two elements. A cyclically ordered set $C$ is a \emph{cycle completion} of $I$ if $I$ is a subset of $C$, the cyclic order on $I$ is the one induced by $C$, and every non-trivial interval of $I$ can uniquely be written as $[v,w]\cap I$ for elements $v,w$ of $C\setminus I$.
\end{defn}

Intuitively, the cycle completion of a cyclic order arises from that cyclic order as follows:
Let $I$ be some cyclically ordered set.
If one envisions $I$ as boxes arranged in a circle according to the cyclic order (see also \cref{fig:cyccomp}), then it is possible to cut up the circle at two places without cutting through boxes, thereby dividing the set of boxes into two intervals.
If $I$ is finite, then every one of these ``cut points'' is between two boxes.
So $I$ and the set of possible cut points form together another cyclically ordered set.
If $I$ is infinite, then not every cut point is between two boxes, but still $I$ and the set of possible cut points form together a cyclically ordered set, the cycle completion of $I$.

Cycle completions can be formalised in several ways, and doing so via cuts, as is done in the appendix, is only one possibility of many.

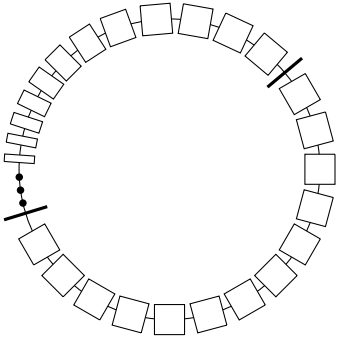
\begin{figure}
\centering
\begin{tikzpicture}
	\draw (0,0) ellipse [radius=2cm];
	\foreach \angle / \width in {
		0/1, 15/1, 30/1,
		50/1, 65/1, 80/1, 95/1,
		110/0.9, 123/0.8, 135/0.7, 145/0.6, 154/0.5, 162/0.4, 169/0.3, 176/0.25,
		345/1, 330/1, 315/1, 300/1, 285/1, 270/1, 255/1, 240/1, 225/1, 210/1
	}
	\path [rotate=\angle,xshift=2cm,yscale=\width, draw=black, fill=white] (0,0) +(0.2,0.2) -- +(0.2,-0.2) -- +(-0.2,-0.2) -- +(-0.2,0.2) -- +(0.2,0.2);
	\foreach \angle in {183, 188, 193}
	\draw (\angle:2cm) node [fill,circle, inner sep=1pt] {};
	\draw [very thick] (40:1.7) -- (40:2.3) (197:1.7) -- (197:2.3);
\end{tikzpicture}\caption{A set (whose elements are indicated by boxes) which is cyclically ordered (indicated by the arrangement of the boxes on a circle) and two ``cutting points'' dividing the cyclically ordered set into two intervals.}\label{fig:cyccomp}
\end{figure}

The cycle completion has several properties which we will use throughout the article and which are intuitively clear.
Also, although we did not find the exact construction of the cycle completion described in the literature, several very similar constructions are.
For these reasons, we will only state the properties needed in this section, and the proofs can be found in the appendix, together with a description of how to obtain the cycle completion from similar constructions.

The first of the facts about cycle completions which we need is that cycle completions exist and are essentially unique.
In order to compare different cyclic orders, we will use \emph{monotone} maps, that is maps $f$ between cyclic orders where $f(v)\in \mathopen]f(u),f(w)\mathclose[$ implies $v\in \mathopen]u,w\mathclose[$, and \emph{isomorphisms of cyclic orders}, bijective maps $f$ where $f(v)\in \mathopen]f(u),f(w)\mathclose[$ is equivalent to $v\in \mathopen]u,w\mathclose[$.
Note that if the image of some monotone map has at least three elements, then the preimage of every interval is also an interval.

\begin{lem}[\cref{uniquecyccomp,fromThomtoTcychom}]\label{existencecyccomp}
For every cyclically ordered set $I$ with at least two elements there is a cycle completion.
If $C$ and $C'$ are two cycle completions of $I$, then there is an isomorphism of cyclic orders $F:C\rightarrow C'$ such that the restriction of $F$ to $I$ is the identity.
\end{lem}

The fact that the cycle completion of a cyclically ordered set $I$ is unique up to isomorphism justifies calling it ``the'' cycle completion of $I$ and denoting it as $C(I)$.
We will also call the elements of $C(I)\setminus I$ \emph{cutpoints}, a terminology partly inspired by the term cut (see appendix).

\begin{lem}[\cref{distinctcutsboundinterval}]\label{distinctvertices}
For a cyclically ordered set $I$ with at least two elements let $v$ and $w$ be distinct elements of $C(I)$.
Then $[v,w]\cap I$ is a non-trivial interval of $I$.
\end{lem}

\begin{lem}[\cref{fromThomtoTcychom}]\label{cyclichoms1}
Let $I$ and $I'$ be cyclically ordered sets with at least two elements and let $f:I'\rightarrow I$ be a surjective monotone map. Then there is a unique surjective monotone map $F:C(I')\rightarrow C(I)$ whose restriction to $I'$ is $f$.
\end{lem}

In particular this implies that the isomorphism from \cref{existencecyccomp} is unique.

\begin{lem}[\cref{fromcyccomphomtoalternatives}]\label{cyclichoms2}
Let $I$ and $I'$ be cyclically ordered sets with at least two elements. Let $F:C(I')\rightarrow C(I)$ be a surjective monotone map with $F(I')\subseteq I$. Then for every $v\in C(I)\setminus I$ there is an element $w$ of $C(I')$ such that $F^{-1}(v)=\{w\}$. This element $w$ satisfies $w\in C(I')\setminus I'$. Also $F(I')=I$.
\end{lem}

So for every surjective monotone map $F:C(I')\rightarrow C(I)$ which satisfies $F(I')=I$ there is an injective monotone map $\hat{f}: C(I)\setminus I \rightarrow C(I')\setminus I'$ such that for all $v\in C(I)\setminus I$ the equation $F^{-1}(v)=\{\hat{f}(v)\}$ holds.

\begin{lem}[\cref{preimageofintervalbounds}]\label{cyclichoms3}
Let $F:C(I')\rightarrow C(I)$ be a surjective monotone map with $F(I')= I$, and $\hat{f}:C(I)\setminus I\rightarrow C(I')\setminus I'$ the injective monotone map with $F\circ \hat{f}(v)=v$ for all $v\in C(I)\setminus I$. Assume that $I$ has at least two elements. Then for all $v,w\in C(I)\setminus I$, $F^{-1}([v,w])=[\hat{f}(v),\hat{f}(w)]$ and $F^{-1}(\mathopen]v,w\mathclose[)=\mathopen]\hat{f}(v),\hat{f}(w)\mathclose[$.
\end{lem}

\section{\texorpdfstring{$k$-Pseudoflowers}{Pseudoflowers}}

In this section we will define $k$-pseudoflowers and prove a few basic facts about them.
This definition is derived from a definition of $k$-flowers in matroids in~\cite{ClarkWhittle}.
A $k$-flower is, essentially, a neat way to encode a collection of separations that mostly cross.
The definition we will give is a lot more complicated than the original one, and there are mainly two reasons for that:
First, a separation of a matroid is a bipartition of the ground set, and thus it is possible to encode several bipartitions in a partition of the ground set.
But in our context we are working with vertex separations, and elements of separators have to be contained in several of the sets corresponding to partition classes.
Second, we work in infinite structures and still want to obtain maximal $k$-pseudoflowers.
Thus we have to make limit processes work, and as a chain of separations of order $k$ may have a limit separation whose order is less than $k$, we have to allow that separations do not have order exactly $k$, but at most $k$.
Making limit processes work also requires to put elements of separators between petals, sometimes in addition to putting them into petals.
Also, every finite cyclic order is determined by its cardinality, but that is not the case for infinite cyclic orders, and thus we have to allow arbitrary cyclic orders as index sets and cannot work with an easily described representative.

Recall that we fixed, for the whole paper, a limit-closed universe of vertex separations on ground set $V$.

\begin{defn}\label{defn:pseudoflower}
Let $I$ be a set of size at least $2$ with a cyclic ordering and $(P_v)_{v\in C(I)}$ a family of vertex sets. Define $X=V\backslash \bigcup_{v\in C(I)} P_v$ and for all $v,w\in C(I)\setminus I$ let $V(v,w)=\bigcup_{z\in [v,w]}P_z\cup X$. Then $(P_i)_{i\in C(I)}$ is a \emph{$k$-pseudoflower} if
\begin{itemize}
\item For every $v\in C(I)\setminus I$ we have $|P_v|=\frac{k-|X|}{2}$;
\item For all distinct $v,w\in C(I)\setminus I$ the pair $S(v,w)=(V(v,w),V(w,v))$ is a separation of order at most $k$ and $V(v,w)\cap V(w,v)=P_v\cup P_w\cup X$;
\item For every $i\in I$ we have $P_{p(i)}\cup P_{s(i)}\subseteq P_i$, where $p(i)$ and $s(i)$ are the predecessor and successor of $i$ in $C(I)$ respectively; and
\item for all $i,i'\in I$,
\begin{displaymath}
\text{if $|P_i|=(k-|X|)/2$ and $P_i=P_{i'}$ then $i=i'$}.\tag{$\ast$}\label{star}
\end{displaymath}
\end{itemize}

A $k$-pseudoflower is \emph{finite} if the index set is finite.
A \emph{$k$-flower} is a $k$-pseudoflower in which all separations $S(v,w)$ are separations with order exactly $k$ such that neither $V(v,w) \setminus V(w,v)$ nor $V(w,v) \setminus V(v,w)$ is empty.
\end{defn}

For $v\neq w\in C(I) \setminus I$ the set $V(v,w)$ is the \emph{interval set} of $[v,w]$ and the separation $S(w,v)$ is the \emph{interval separation} of $[v,w]$.
For a non-trivial interval $I'\subseteq I$ let $v$ and $w$ be the unique elements of $C(I)\setminus I$ such that $I'=I\cap [v,w]$. Then the \emph{interval set} $V(I')$ of $I'$ is $V(v,w)$ and the \emph{interval separation} $S(I')$ of $I'$ is $S(w,v)$.
For each index $i\in I$ the \emph{petal} of $i$ is its interval set $V(\{i\})$ and the \emph{petal separation} of $i$ is its interval separation $S(\{i\})$. As is usual for maps defined on power sets, we shorten $V(\{i\})$ and $S(\{i\})$ to $V(i)$ and $S(i)$.
Note that $V(i)=P_i \cup X$ for all $i\in I$.
A $k$-pseudoflower is called a \emph{$k$-pseudoanemone} if the sets $P_v$ for $v\in C(I) \setminus I$ are all empty and a \emph{$k$-pseudodaisy} otherwise.
Note that the number $(k-|X|)/2$, which is the size of the sets $P_v$ with $v\in C(I)\setminus I$, indicates how far away a $k$-pseudodaisy is from being a $k$-pseudoanemone.
In this sense it plays a very similar role to the difference between the local connectivity of adjacent petals and the local connectivity of non-adjacent petals as used, for example, in \cite{AikinOxley}.
Also note that in a $k$-pseudoanemone, all the interval separations have $X$ as their separator.
In the case where the limit-closed universe comes from a graph $G$, this means that the set of petals corresponds to a partition of the components of $G-X$, and that $X$ has exactly $k$ elements.

\begin{rem}
In a graph, $k$-pseudoanemones correspond to partitions of the components of $G-X$.
Also there are no infinite $k$-pseudodaisies in graphs. See \cref{ex:infinitedaisy} for an example of an infinite $k$-pseudodaisy in a graph-like space.
\end{rem}

\subsection{Basic properties}

We start with a lemma that is, among other things, essential for determining separations of the form $S(I') \vee S(I'')$.

\begin{figure}
	\centering
	\begin{tikzpicture}
	\basicflower
	\orientation
	\labelofarcabove[blue]{90}{0}{1}{$[a,b]$}
	\labelofarcabove[picturegreen]{270}{180}{1}{$[c,d]$}
	\newnode{$a$}{90}
	\newnode{$b$}{0}
	\newnode{$c$}{-90}
	\newnode{$d$}{180}
	\end{tikzpicture}
	\caption{A cyclic order with two intervals. This figure depicts some of the notation used in \cref{intersectionofintervalsets,lemmaforV}.}
	\label{fig:intersectionofintervalsets}
\end{figure}
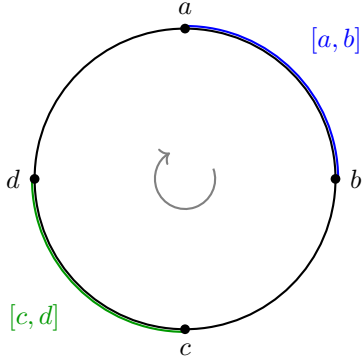

\begin{lem}\label{intersectionofintervalsets}
	Let $\Phi$ be a $k$-pseudoflower on index set $I$.
	Let $a$, $b$, $c$ and $d$ be elements of $C(I)\setminus I$ such that $a$, $b$ and $c$ are pairwise distinct and $b\in [a,c]$.
	Also assume that if $d\in \{a,b,c\}$, then $d=a$, and if $d\notin \{a,b,c\}$ then $d\in [c,a]$.
	Then $V(a,b)\cap V(c,d)=(P_a\cap P_d)\cup (P_b\cap P_c)\cup X$ and $V(a,c)\cap V(b,d)=V(b,c)\cup (P_a\cap P_d)$.
\end{lem}
\begin{proof}
	Clearly $V(a,b)\cap V(c,d)\subseteq V(a,c)\cap V(c,a)=P_a\cup P_c\cup X$.
	Similarly $V(a,b)\cap V(c,d)$ is a subset of $P_a\cup P_b\cup X$, $P_d\cup P_b\cup X$ and $P_d\cup P_c\cup X$.
	Together these subsetrelations imply
	\begin{equation*}
	\begin{split}
	&V(a,b)\cap V(c,d)\\
	&\quad \subseteq (P_a\cup P_b\cup X)\cap (P_a\cup P_c\cup X)\cap (P_d\cup P_b\cup X)\cap (P_d\cup P_c\cup X)\\
	&\quad =(P_a\cap P_d)\cup (P_b\cap P_c)\cup X.
	\end{split}
	\end{equation*}
	As also $(P_a\cap P_d)\cup (P_b\cap P_c)\cup X\subseteq V(a,b)\cap V(c,d)$, these two sets are equal.
	Thus
	\begin{equation*}
	\begin{split}
	V(a,c)\cap V(b,d)&=V(b,c)\cup (V(a,b)\cap V(c,d))\\
	&=V(b,c)\cup (P_a\cap P_d)\cup (P_b\cap P_c)\cup X\\
	&=V(b,c)\cup (P_a\cap P_d).\qedhere
	\end{split}
	\end{equation*}
\end{proof}

\begin{cor}\label{meetofintervalseps}
	If $P_a\cap P_d=\emptyset$ then $S(b,c)=S(a,c)\wedge S(b,d)$.\qedhere
\end{cor}

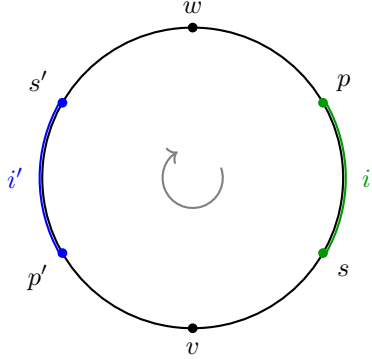
\begin{figure}
\centering
\begin{tikzpicture}
    \basicflower
    \orientation
    \newnode{$w$}{90}
    \newnode{$v$}{-90}
    \newnode[blue]{$p'$}{210}
    \newnode[blue]{$s'$}{150}
    \newnode[picturegreen]{$p$}{30}
    \newnode[picturegreen]{$s$}{-30}
    \labelofarcabove[blue]{210}{150}{1}{$i'$}
    \labelofarcabove[picturegreen]{30}{-30}{1}{$i$}
\end{tikzpicture}
    \caption{Some notation from the proof of \cref{flowerhasdisjointgluingsets}.}
    \label{fig:flowerhasdisjointgluingsets}
\end{figure}

The next lemma and the previous corollary will often be used together.

\begin{lem}\label{flowerhasdisjointgluingsets}
Let $(P_i)_{i\in C(I)}$ be a $k$-flower on index set $I$.
Then $P_v\cap P_w=\emptyset$ for all $v\not=w \in C(I)\setminus I$.
\end{lem}
\begin{proof}
By definition of a $k$-flower, the set $P_v\cup P_w \cup X$, which equals $V(v,w)\cap V(w,v)$, has exactly $k$ many elements.
Also $P_v$ and $P_w$ both have exactly $(k-|X|)/2$ many elements, so they must be disjoint.
\end{proof}

\begin{lem}\label{verticesinducesintervalsofV(C)}
For all $x\in \bigcup_{v\in C(I)\setminus I}P_v$ the set $\{w\in C(I):x\in P_w\}$ is an interval of $C(I)$.
\end{lem}
\begin{proof}
Assume otherwise. Then there are pairwise distinct $a,b,c,d\in C(I)$ such that $b\in [a,c]$, $d\in [c,a]$, $x\in P_b\cap P_d$ and $x\notin P_a\cup P_c$.
If $a \in I$ then let $a'$ be the predecessor of $a$ in $C(I)$, otherwise let $a' = a$.
Similarly, if $c \in I$ then let $c'$ be the predecessor of $c$ in $C(I)$, otherwise let $c' = c$.
Then $x \notin P{a'} \cup P_{c'}$, and $a',b,c',d$ are distinct elements of $C(I) \setminus I$.
Now
\begin{displaymath}
    x \in P_b \cap P_d \subseteq V(a',c') \cap V(c',a') = P_{a'}\cup P_{c'} \cup X,
\end{displaymath}
a contradiction.
\end{proof}

\subsection{\texorpdfstring{Concatenations of $k$-pseudoflowers}{Concatenations of pseudoflowers}}\label{concatenations}

A separation $(A,B)$ is \emph{displayed} by a $k$-pseudoflower if it is an interval separation of that $k$-pseudoflower.
We are now going to define a pre-order on the set of $k$-pseudoflowers which has the property that if $\Phi\leq \Psi$ then all separations displayed by $\Phi$ are also displayed by $\Psi$.

\begin{defn}
A $k$-pseudoflower $\Phi'$ \emph{extends} another $k$-pseudoflower $\Phi$, written $\Phi\leq \Phi'$, if the sets $X$ and $X'$ coincide and there is a surjective map $F:C(I')\rightarrow C(I)$ respecting the cyclic ordering such that $F(I')= I$ and $V(v,w)= X\cup \bigcup_{z\in F^{-1}([v,w])}P'_z$ for all $v,w\in C(I)\setminus I$.
If $\Phi'$ extends $\Phi$, then $\Phi$ is a \emph{concatenation} of $\Phi'$.
\end{defn}

Note that if $\Phi \leq \Psi \leq \Phi$ for two $k$-pseudoflowers $\Phi$ and $\Psi$ then property (\ref{star}) $\Phi$ and $\Psi$ are equal up to renaming of the index sets.

Given a map $F:C(I')\rightarrow C(I)$ that witnesses that $\Phi\leq \Phi'$, by \cref{cyclichoms2} there is for every $v\in C(I)\setminus I$ a unique $w\in C(I')$ with $F(w)=v$.
Furthermore $w\in C(I')\setminus I'$, and $P_v=P_w$ by definition of $\Phi\leq \Phi'$.
This allows us to identify $F^{-1}(C(I')\setminus I')$ with $C(I')\setminus I'$, and thus view $\Phi$ as a substructure of $\Phi'$.
(See \cref{sec:limits} for a more detailed explanation of how this identification can be done for a whole chain of $k$-pseudoflowers.)
Note that the witness of $\Phi\leq \Phi'$ might not be unique, and that the identification of cutpoints depends on the witness.
Conversely, given a finite set $D$ of cutpoints of a $k$-pseudoflower $\Phi'$ on index set $I'$, we can take the subflower $\Phi'(D)$ of $\Phi'$ whose cutpoints correspond to $D$ as follows:
Let $I$ be a cyclic order with $|D|$ many elements, and let $\hat{f}: C(I)\setminus I \rightarrow D$ be an isomorphism of cyclic orders.
Extend the inverse of $\hat{f}$ to a surjective monotone map $F:C(I')\rightarrow C(I)$ with $F(I')\subseteq I$.
For $v\in C(I)\setminus I$ there is a unique $d\in D$ with $F^{-1}(v)=d$, define $P_v=P_d$.
For $i\in I$ let $p$ and $s$ be the predecessor and successor of $i$ in $C(I)$ respectively, and define $P_i=\bigcup_{z\in F^{-1}([p,s])}P_z$.
Then $(P_z)_{z\in C(I)}$ is a $k$-pseudoflower, and $F$ witnesses $(P_z)_{z\in C(I)}\leq \Phi'$.
Note that we need this construction only for finite $D$, but it works for all sets of cutpoints $D$ that can be expressed as $C(I)\setminus I$ for some cyclically ordered set $I$.

\begin{lem}\label{intervalcontainingset}
Let $\Phi$ be a $k$-pseudoflower that extends a $k$-flower with three petals.
Then for every $z\in C(I)$ there is $v\in C(I)\setminus I$ such that $P_z\cap P_v=\emptyset$.
Furthermore, for every $i\in I$ and all non-empty $Y\subseteq P_i$ the set $\{z\in C(I) : Y\subseteq P_z\}$ is a non-trivial interval of $C(I)$.
\end{lem}
\begin{proof}
As $\Phi$ extends a $k$-flower with three petals, $C(I)\setminus I$ contains three distinct elements $v_1$, $v_2$ and $v_3$ such that $P_{v_1}$, $P_{v_2}$ and $P_{v_3}$ are pairwise disjoint.
By renaming if necessary assume that $z\in [v_1,v_2]$ and $v_3\in [v_2,v_1]$.
Then
\begin{displaymath}
    P_z \cap P_{v_3} \subseteq V(v_1,v_2) \cap V(v_2,v_1) = P_{v_1} \cup P_{v_2} \cup X.
\end{displaymath}
As every $P_y$ with $y\in C(I)$ is disjoint from $X$, this implies that $P_z\cap P_{v_3}=\emptyset$.

Let $Y\subseteq P_i$ and for every $y\in P_i$ let $C_y$ be the interval of those $z\in C(I)$ with $y\in P_z$.
By the previous paragraph there exists $v\in C(I)\setminus I$ such that $P_i\cap P_v=\emptyset$.
As all $C_y$ contain $i$ but none contains $v$, the intersection of the $C_y$ is a non-trivial interval of $C(I)$.
\end{proof}

We already stated that given two $k$-pseudoflowers $\Phi\leq \Psi$, there might be several maps witnessing the relation.
It may not be obvious from the definition, but if $\Phi$ is a sufficiently meaningful $k$-pseudoflower, then two witnessing maps can only differ in very restricted ways.
The rest of this section contains an analysis of the exact ways in which they can differ.

For $k$-pseudoflowers $\Phi\leq \Phi'$ such that $\Phi$ extends a $k$-flower with three petals let $I'_N$ be the set of those $i'\in I'$ such that there are witnesses of $\Phi\leq \Phi'$ that map $i'$ to different elements of $I$.

\begin{lem}\label{nonuniqueimagesunderwitnesses}
Let $i'\in I_N'$.
Then the set of images of $i'$ under witnesses of $\Phi\leq \Phi'$ contains only two elements of $I$, and those have a common neighbor $v(i')$ in $C(I)$ that satisfies $P_{i'}'=P_{v(i')}$.
\end{lem}
\begin{proof}
In order to show that $P'_{i'}=P_v$ for some $v\in C(I)\setminus I$, let $i_1\neq i_2$ be possible images of $i'$ under witnesses of $\Phi\leq\Phi'$.
Then in particular $P'_{i'}\subseteq P_{i_1}$ and $P'_{i'}\subseteq P_{i_2}$.
By \cref{intervalcontainingset} there is $v_1\in C(I)\setminus I$ such that $P_{v_1}\cap P_{i_1}=\emptyset$.
Without loss of generality $v_1\in [i_1,i_2]$, and as $i_1\neq i_2$ there is $v_2\in [i_2,i_1]\setminus I$.
So
\begin{displaymath}
P'_{i'}\subseteq P_{i_1}\cap P_{i_2}\subseteq V(v_1,v_2)\cap V(v_2,v_1)=P_{v_1}\cup P_{v_2}\cup X,
\end{displaymath}
implying that $P'_{i'}\subseteq P_{v_2}$ and thus by first and third condition of $k$-pseudoflowers that $|P'_{i'}|=|P_{v_2}|=(k-|X|)/2$.

If there is $i\in I$ with $P_i=P'_{i'}$, then every witness $F$ of $\Phi\leq \Phi'$ has to map some $i_F'$ to $i$, and $P'_{i_F'}=P_i=P'_{i'}$ then implies by (\ref{star}) that $i_{F'}=i'$.
Thus in this case every witness of $\Phi\leq \Phi'$ maps $i'$ to $i$, contradicting $i'\in I'_N$.

Let $F$ be some witness of $\Phi\leq \Phi'$, then $P'_{i'}\subseteq P_{F(i')}$.
By \cref{intervalcontainingset} there is $z_1\in C(I)\setminus I$ such that $P_{F(i')}\cap P_{z_1}=\emptyset$, and the set $C:=\{z\in C(I): P'_{i'}\subseteq P_z\}$ is a non-trivial interval of $C(I)$.
There is $z_1'\in C(I')\setminus I'$ with $F(z_1')=z_1$.

Assume that there is $i\in C\cap I$ that is (in the linear order of the interval $C$) neither the biggest nor the smallest element of $C\cap I$.
Then the two neighbors $s$ and $p$ of $i$ in $C(I)$ are also contained in $C$, thus $P_s=P_p=P'_{i'}\neq P_i$.
Hence there is $x\in P_i$ with $x\notin P_z$ for all other $z\in C(I)$, and there is $z_2\in C(I')$ with $x\in P_{z_2}$.
Now $F(z_2)=i$.
If $i'\in \mathopen]z_1',z_2\mathclose[$ then $F(i')\in \mathopen]z_1,i]$ and otherwise $i'\in \mathopen]z_2,z_1'\mathclose[$, implying $F(i')\in [i, z_1]$.
Both $\mathopen]z_1,i]\cap C$ and $[i,z_1\mathclose[\cap C$ are intervals of $C$.

As $F$ was chosen as an arbitrary witness of $\Phi\leq \Phi'$, this implies that $C\cap I$ contains an interval of at most two elements that contains all images of $i'$ under witnesses of $\Phi\leq \Phi'$.
As $i'\in i_N'$ this implies that there are exactly two such images, and that they are neighbors in $I$ and hence have a common neighbor $v(i')$ in $C(I)$.
Then $v(i')\notin I$, and
\begin{displaymath}
    P'_{i'}\subseteq V(v_1,v(i'))\cap V(v(i'),v_1)=X\cup P_{v_1}\cup P_{v(i')}.
\end{displaymath}
As $P'_{i'}$ is disjoint from $X$ and $P_{v_1}$ and is at least as big as $P_{v(i')}$, this implies $P'_{i'}=P_{v(i')}$.
\end{proof}

\begin{lem}\label{witnessexistspreimage}
For every $i\in I$ there is some $i'\in I'\setminus I'_N$ such that every witness of $\Phi\leq \Phi'$ maps $i'$ to $i$.
\end{lem}
\begin{proof}
Let $I'_i$ be the set of those elements of $I'$ that are mapped to $i$ by some witness of $\Phi\leq\Phi'$.
If $I'_i$ contains only one element, then the lemma holds, so assume otherwise.
Assume for a contradiction that all elements of $I'_i$ are contained in $I'_N$.
Denote the neighbors of $i$ in $C(I)$ by $s$ and $p$, then by \cref{nonuniqueimagesunderwitnesses} every element $i'\in I'_i$ satisfies either $P_{i'}=P_s$ or $P_{i'}=P_p$, hence there are exactly two such elements $i_1$ and $i_2$.
Again by \cref{nonuniqueimagesunderwitnesses}, $P_{i_1}$ and $P_{i_2}$ have both the same size as all sets $P_v$ with $v \in C(I) \setminus I$.
But if $v$ denotes the common neighbour of $i_1$ and $i_2$ in $C(I')$, then $P_v \subseteq P_{i_1} \cap P_{i_2}$, implying $P_{i_1} = P_{i_2}$ and thus contradicting property (\ref{star}).
\end{proof}

\begin{lem}
    For every $v\in C(I)\setminus I$ that is not of the form $v(i')$ for some $i'\in I'_N$ there is a unique $v'\in C(I')\setminus I'$ that is mapped to $v$ by all witnesses of $\Phi\leq \Phi'$.
\end{lem}
\begin{proof}
    As $\Phi$ extends a $k$-flower with three petals, there is $i\in I$ such that $P_v\cap P_i=\emptyset$ and $v$ is not a neighbor of $i$.
    Denote the predecessor and successor of $i$ in $C(I)$ by $p$ and $s$ respectively.
    In particular $p\neq v\neq s$.
    Let $i'$ be an element of $I'$ that is mapped to $i$ by every witness that $\Phi\leq \Phi'$ and denote its successor in $C(I')$ by $s'$.

    Assume that there are witnesses $F_1$ and $F_2$ of $\Phi\leq \Phi'$ such that $F_1^{-1}(v)=v_1$ and $F_2^{-1}(v)=v_2$ and $v_1\in \mathopen]s,v_2\mathclose[$.
    Then for $j\in \{1,2\}$, $F_j^{-1}(s)\in [s',v_j]$ and thus $V(s',v_2)\subseteq V(p,v)$ and $V(v_1,s')\subseteq V(v,s)$.
    Hence every element of $V(v_1,v_2)$ is contained in $V(i)\cup V(s,v)$ and in $V(v,s)$.
    But $V(v,s)\cap (V(i)\cup V(s,v))=V(i) \cup P_v$ and $V(v_1,v_2)\cap (V(i)\cup P_v)\subseteq V(v_1,v_2)\cap V(v_2,v_1)=X\cup P_v$, hence every element of $V(v_1,v_2)$ is contained in $X\cup P_v$.
    So $P_{\hat{\imath}}=P_v$ for all $\hat{\imath}\in [v_1,v_2]$ and in particular there can be at most one such $\hat{\imath}$.
    
    As $v_1\neq v_2$, some such $\hat{\imath}$ exists.
    Then $F_2(\hat{\imath})\in [s,v]$ and, because $P_{\hat{\imath}}\cap P_i=P_v\cap P_i=\emptyset$, $F_1(\hat{\imath})\in [v,p]$.
    So $\hat{\imath}\in I'_N$ and $v(\hat{\imath})=v$, contradicting the assumptions about $v$.
\end{proof}

Let $F$ be a witness that $\Phi\leq \Phi'$ and $f$ the restriction of $F$ to $I'$.
Let $f':I'\rightarrow I$ be obtained from $f$ by choosing, for every $i'\in I'_N$, some neighbor of $v(i')$ in $C(I)$ as the image of $f'$ and mapping every other element of $I'$ to its image under $f$.
Let $F':C(I')\rightarrow C(I)$ be the unique surjective map respecting the cyclic order whose restriction to $I'$ is $f'$ (see \cref{cyclichoms1}).

\begin{lem}\label{modifywitnesses}
The map $F'$ is well-defined.
\end{lem}
\begin{proof}
    It suffices to show that $f'$ is surjective and respects the cyclic order.
    As for every $i\in I$ there is some $i'\in I'$ that is mapped to $i$ by all witnesses of $\Phi\leq \Phi'$ and is in particular not contained in $I_N'$, $f'$ is indeed surjective.
    
    In order to show that $f'$ respects the cyclic order, consider elements $i_1$, $i_2$ and $i_3$ of $I'$ such that $f'(i_2)\in \mathopen]f'(i_1),f'(i_3)\mathclose[$.
    In order to show that $i_2\in \mathopen]i_1,i_3\mathclose[$ it suffices to consider the case that $f'$ differs from $f$ in at most three elements of $I$, and for that it is enough to show that $f'$ respects the cyclic order if it maps only one element of $I'$ differently than $f$.
    So assume that $f=f'$ except that $f'(i_1)$ is the successor of $f(i_1)$, the other cases are symmetric.
    Let $i_1'$ be an element of $I'$ that is mapped to $f'(i_1)$ by all witnesses of $\Phi\leq \Phi'$.
    In particular $f(i_1')=f'(i_1')$, and $f(i_1)$ is the predecessor of $f(i_1')$.
    If $f(i_1)\neq f(i_3)$, then $i_2\in \mathopen]i_1,i_3\mathclose[$ because $f$ respects the cyclic order.
    Otherwise $f^{-1}(f(i_1))$ is an interval of $I'$ containing both $i_1$ and $i_3$ but neither $i_2$ nor $i_1'$, and $f'^{-1}(f'(i_1))$ is an interval of $I'$ containing both $i_1$ and $i_1'$ but neither $i_2$ nor $i_3$.
    Together with $i_2\in \mathopen]i_1,i_3\mathclose[$ this implies that $i_1\in \mathopen]i_3,i_1'\mathclose[$ and hence that $i_1\in \mathopen]i_3,i_2\mathclose[$.
\end{proof}

\begin{lem}\label{newwitness}
The map $F'$ is a witness of $\Phi\leq \Phi'$.
\end{lem}
\begin{proof}
Let $v$ and $w$ be distinct elements of $C(I)\setminus I$.
If $F'^{-1}([v,w])$ differs from $F^{-1}([v,w])$, then it does so in elements $i$ of $I_N'$ with $v(i)=v$ or $v(i)=w$ and one of their neighbors.
Thus
\begin{displaymath}
X\cup \bigcup_{z\in F^{-1}([v,w])} P_z'=X\cup \bigcup_{z\in F'^{-1}([v,w])} P_z'.\qedhere
\end{displaymath}
\end{proof}

\subsection{\texorpdfstring{Tangles and profiles in $k$-pseudoflowers}{Tangles and profiles in pseudoflowers}}

The theory presented in this paper will be presented for two very similar objects: Profiles and tangles.
The main results hold for the more general object, namely profiles.
Thus they also hold for tangles, and in this case the proof for tangles is shorter.
As the term tangle is more widely known than the term profile, the main results will be presented for both tangles and profiles, with the main differences to be found in this subsection.
In the rest of the paper, we will simply work with a set $\mathcal{P}$ of $k+1$-profiles, and as every $k+1$-tangle is also a $k+1$-profile the reader is free to think of a set of tangles.

The following definitions are again to be found in \cite{ASS}.
Given a separation systems $S_k$ of a universe of vertex separations, a \emph{consistent orientation} of $S_k$ is a set $O \subseteq S_k$ such that for all $s \in S_k$, exactly one element of the set $\{s,s^*\}$ is contained in $O$, and such that for two elements $r,s \in O$ if $r^* \leq s$ then $\{r,r^*\} = \{s,s^*\}$.
A \emph{$k$-profile} of $S_k$ is a consistent orientation $P$ such that for all elements $r$ and $s$ of $P$ $(r\vee s)^*\notin P$ holds.
A $k$-tangle of $S_k$ is a consistent orientation $T$ of $S_k$ such that for all elements $(A_1,B_1)$, $(A_2,B_2)$ and $(A_3,B_3)$ of $T$, $(A_1,B_1) \vee (A_2,B_2) \vee (A_3,B_3)$ is not cosmall, which for universes of vertex separations means that $A_1 \cup A_2 \cup A_3 \neq V$.

Given a graph, a $k$-profile of that graph is defined to be a profile of $S_k$ in the universe of vertex separations of $G$.
So it is obvious what the definition of a $k$-profile of a graph-like space or a limit-closed universe of vertex separations should be: A $k$-profile of the universe of all vertex separations.
Similarly, a $k$-tangle of a limit-closed universe of vertex separations is a tangle of $S_k$.

Readers interested in the technical details might note that if the universe is the universe of vertex separations of a graph, then this definition of a $k$-tangle is not exactly the same as the definition of a $k$-tangle of a graph.
The two definitions are very close, however, and every $k$-tangle of the graph is also a $k$-tangle of the universe. 

For finite $k$-flowers, a $(k+1)$-tangle points towards a unique petal $i$ of the flower in the sense that the $(k+1)$-tangle contains $S(i)$.
A slightly weaker form of this property also holds in infinite $k$-pseudoflowers.
Given a $k$-pseudoflower $\Psi$ on index set $I$ a $k+1$-profile $P$ \emph{is located by $\Psi$} if there is some $v \in C(I) \setminus I$ such that either $\forall w\in C(I)\setminus I -v:S(v,w)\in P$ or $\forall w\in C(I)\setminus I - v:S(w,v)\in P$.

For tangles we always have this property.
\begin{lem}\label{placetanglesinflower}
Let $\Phi$ be a $k$-pseudoflower and let $T$ be a $k+1$-tangle.
Then $T$ is located by $\Phi$.
\end{lem}
\begin{proof}
Let $I$ be the index set of $\Phi$ and let $u\in C(I)\setminus I$. Let $V'$ be the set of all elements $w$ of $(C(I)\setminus I)-u$ such that $S(u,w)\in T$.
If $V'$ is empty or $(C(I)\setminus I)-u$, then $u$ witnesses that $T$ is located by $\Phi$, so assume otherwise.
As $T$ is consistent, $V'$ is an interval of $C(I)\setminus I$ and thus of the form $\mathopen]u,v\mathclose[\cap C(I)\setminus I$ or $\mathopen]u,v]\cap C(I)\setminus I$ for some $v\in C(I)\setminus I$.
First consider the case that $V'$ is of the form $\mathopen]u,v\mathclose[\cap C(I)\setminus I$ for some $v\in C(I)\setminus I$.
Then $S(u,v)$ is not an element of $T$ and thus $S(v,u)\in T$.
Hence for all $w\in \mathopen]v,u\mathclose[$ we have $S(v,w)\in T$ by consistency of $T$.
Also for all $w\in \mathopen]u,v\mathclose[$ we have that $S(u,w)\in T$ so because also $S(v,u)\in T$ the tangle property implies $S(v,w)\in T$.
So in this case we are done.
Now consider the case that $V'$ is of the form $\mathopen]u,v]$ for some $v\in C(I)\setminus I-u$.
Then $S(u,v)\in T$, so also $S(w,v)\in T$ for all $w\in [u,v\mathclose[$ by consistency.
Furthermore for all $w\in \mathopen]v,u\mathclose[$ we have $S(u,w)\notin T$, so $S(w,u)\in T$ and thus $S(w,v)\in T$ by the tangle property.
So in this case we are done as well.
\end{proof}

For general $k+1$-profiles, the statement holds with an additional assumption that the $k$-pseudoflower is sufficiently large.

\begin{figure}
\centering
\begin{tikzpicture}
	\basicflower
	\orientation
    \labelofsep{90}{0}{$S(t,s)$\\$\in P$}{0}
	\newnode{$s$}{90}
	\newnode{$t$}{0}
	\newnode{$x$}{-90}
	\newnode{$y$}{180}
	\newnode{$v$}{45}
\end{tikzpicture}
\caption{$S(t,s)$ is an interval separation of a $k$-pseudoflower and is contained in a $k+1$-profile $P$. See also \cref{profilelocatedinpseudoflower}.}
\label{fig:profilelocatedinpseudoflower}
\end{figure}
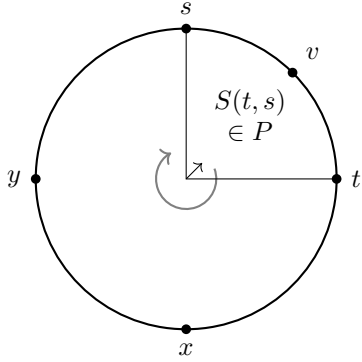

\begin{lem}\label{profilelocatedinpseudoflower}
Let $\Phi$ be a $k$-pseudoflower which extends a $k$-flower with at least four petals and let $P$ be a $k+1$-profile.
Then $P$ is located by $\Psi$.
\end{lem}
\begin{proof}
Let $I$ be the index set of $\Psi$. Let $s$, $t$, $x$ and $y$ be elements of $C(I)\setminus I$ such that $\Phi(\{s,t,x,y\})$ is a $k$-flower with four petals and $t\in[s,x]$ and $y\in [x,s]$.
    Then $P$ contains an orientation of $S(s,x)$ and an orientation of $S(t,y)$, and by \cref{meetofintervalseps,flowerhasdisjointgluingsets} $P$ also contains the supremum of those two orientations.
    So $P$ contains the inverse of a petal separation of $\Phi(\{s,t,x,y\})$, without loss of generality assume that $P$ contains $S(t,s)$.

	In the interval $[s,t]$ let $v$ be the supremum of $\{w\in [s,t\mathclose[\setminus I\colon S(t,w)\in P\}$.
	There are three cases: $v=t$, $S(v,x)\in P$ and $S(x,v)\in P$.
	In the first case we have $\forall w\in C(I)\setminus I-v:S(v,w)\in P$ and are done.
	
	Consider the second case: $v\neq t$ and $S(v,x)\in P$.
	Then $S(v,w)=S(v,x)\vee S(t,w)$ for all $w\in [s,v\mathclose[$ by \cref{meetofintervalseps,flowerhasdisjointgluingsets} and thus $S(v,w)\in P$.
	Hence $\forall w\in C(I)\setminus I-v:S(v,w)\in P$.
	
	In the last case, $v\neq t$ and $S(x,v)\in P$.
	Again $S(t,v)=S(t,s)\vee S(x,v)$ and thus $S(t,v)\in P$.
	If $t$ is the successor of $v$ in $C(I)\setminus I$, then $S(t,v)\in P$ implies that $\forall w\in C(I)\setminus I-v:S(w,v)\in P$.
	Otherwise let $w\in \mathopen]v,t\mathclose[\setminus I$.
	By the definition of $v$, $P$ does not contain $S(t,w)$, so $S(t,w)=S(t,s)\vee S(y,w)$ implies that $S(y,w)$ is not an element of $P$.
	Thus $P$ contains $S(w,y)$, and $S(w,v)=S(w,y)\vee S(x,v)$ implies that $S(w,v)\in P$.
	As this is true for all $w\in \mathopen]v,t\mathclose[\setminus I$, $\forall w\in C(I)\setminus I-v:S(w,v)\in P$ holds.
\end{proof}

For the remainder of this article, instead of differentiating between tangles and profiles, we will always work more generally with profiles located by the relevant flowers. Since, as noted above, \cref{placetanglesinflower} also works for tangles of graphs, our approach encompasses this case as well.

We want to have $k$-pseudoflowers that distinguish as many elements of $\mathcal{P}$ as possible in the following sense:
We call two $\leq k$-separations \emph{equivalent (with respect to $\mathcal{P}$}) if they are are contained in the same elements of $\mathcal{P}$.
An equivalence class $E$ of $\leq k$-separations is \emph{displayed} by a $k$-pseudoflower $\Phi$ if an element of $E$ is displayed by $\Phi$.
We are interested in $k$-pseudoflowers displaying as many equivalence classes as possible.
For that, we also consider the following relation on $k$-pseudoflowers.
Note that it is not a partial order as it is not antisymmetric.

\begin{defn}
Let $\Phi$ and $\Phi'$ be two $k$-pseudoflowers. Define $\Phi\preccurlyeq \Phi'$ if and only if every relevant separation displayed by $\Phi$ is equivalent to a separation displayed by $\Phi'$. If $\Phi\preccurlyeq \Phi'\preccurlyeq \Phi$, then the $k$-pseudoflowers are \emph{equivalent}.
\end{defn}

\section{Extending flowers}
Two separations \emph{properly cross} if all of their corner separations are contained in profiles of $\mathcal{P}$.

Let $\Psi$ be a $k$-pseudoflower and $i$ a petal of $\Psi$ with predecessor $p$ and successor $s$. A separation $(C,D)$ properly crossing $S(i)$ is \emph{anchored} at $v \in C(I)\setminus I$ if $(C,D)\wedge S(i)=S(v,p)$ and $(D,C)\wedge S(i)=S(s,v)$.
\begin{lem}\label{findanchoringset}
    Let $\Phi$ be a $k$-pseudoflower distinguishing at least three elements of $\mathcal{P}$ located by $\Phi$.
    Let $(C,D)$ be a separation of order $k$ which properly crosses some petal separation $S(i)$ of $\Phi$.
    Then there is a separation $(C',D')$ properly crossing $S(i)$ such that
    \begin{itemize}
        \item $(C,D)\vee S(i)=(C',D')\vee S(i)$ and $(D,C)\vee S(i)=(D',C')\vee S(i)$.
        \item $(C',D')$ or its inverse is anchored at some $v\in C(I)\setminus I$.
    \end{itemize}
\end{lem}
\begin{proof}
    Let $P_1 \in \mathcal{P}$ be a profile which contains both $(C,D)$ and $S(i)$.
    If $\Phi$ distinguishes two profiles which contain both $(C,D)$ and $S(i)^*$ then let $P_2 \in \mathcal{P}$ be a profile which contains both $(D,C)$ and $S(i)^*$.
    Then some profile $P_3 \in \mathcal{P}$ which contains both $(C,D)$ and $S(i)^*$ is distinguished from $P_2$ by $\Phi$.
    If $\Phi$ does not distinguish any profiles which contain both $(C,D)$ and $S(i)^*$, then let $P_3$ be some such profile.
    As $\Phi$ distinguishes at least three profiles, it distinguishes some profile $P_2$ from both $P_1$ and $P_3$.
    Because $P_2$ is distinguished from $P_1$, it contains $S(i)^*$.
    Because $P_2$ is also distinguished from $P_3$, it does not contain $(C,D)$, so it contains $(D,C)$.
    In both cases, $P_2$ contains $S(i)^*$ and $(D,C)$, while $P_3$ contains $S(i)^*$ and $(C,D)$.
    Furthermore $\Phi$ distinguishes $P_2$ and $P_3$.
    
    Denote the predecessor of $i$ in $C(I)$ by $p$ and the successor by $s$.
    As $\Phi$ distinguishes $P_2$ and $P_3$, there is a cutpoint $v\in C(I)\setminus I$ such that $S(v,s)$ distinguishes $P_2$ and $P_3$.
    Because both $P_2$ and $P_3$ contain $S(i)^*$, $S(v,p)$ also distinguishes $P_2$ and $P_3$, and $P_2$ contains $S(v,s)$ if and only if it contains $S(v,p)$.
    
    From here on consider the case that $P_3$ contains $S(v,p)$, the other case is symmetric.
    Define $(C'',D'')=(C,D)\vee S(v,p)$.
    Because $P_3$ contains both $(C,D)$ and $S(v,p)$ and $P_2$ contains both $S(p,v)$ and $(D,C)$, the separation $(C'',D'')$ has order $k$ and distinguishes $P_2$ and $P_3$.
    Then $(C'',D'')\vee S(i)=(C,D)\vee S(v,p)\vee S(i)=(C,D)\vee S(i)$ and in particular $(C'',D'')\vee S(i)$ is the same as $(C,D)\wedge S(i)$.
    Let $P_4$ be a profile that contains both $S(i)$ and $(D,C)$.
    Then $(D'',C'')\leq (C,D)$ implies that $(D'',C'')$ is contained in $P_4$.
    The profile $P_3$ contains both $(C'',D'')$ and $S(i)^*$ and the profile $P_4$ contains both $(D'',C'')$ and $S(i)$, so $(D'',C'')\vee S(i)$ is a separation of order $k$.
    Also $(D'',C'')\vee S(i)\leq (D,C)\vee S(i)$. Because $V(v,p)\cup V(s,p)=V$, we have $D\cup V(s,p)=(D\cap V(p,v))\cup V(s,p)$. The set $D\cup V(s,p)$ is the left side of $(D,C)\vee S(i)$ and $(D\cap V(p,v))\cup V(s,p)$ is the left side of the separation $(D'',C'')\vee S(i)$. Also, the right side of $(D,C)\vee S(i)$ is a subset of the right side of $(D'',C'')\vee S(i)$. Because these two separations have the same order, they are equal.
    
    Define $(D',C')=(D'',C'')\vee S(s,v)$. Just as in the previous paragraph $(C',D')$ is a separation of order $k$ distinguishing $P_2$ and $P_3$.
    Also $(D',C')\vee S(i)=(D'',C'')\vee S(i)$ and $(C',D')\vee S(i)=(C'',D'')\vee S(i)$.
    Furthermore $P_1$ contains $(C',D')$ and $S(i)$, and $P_3$ contains $(D',C')$ and $S(i)^*$, so $(C',D')\wedge S(i)$ has order $k$.
    Because $S(v,p)\leq (C',D')\leq S(v,s)$, also
    \begin{displaymath}
    (V(v,p)\cup P_s, V(p,v)) = S(v,s)\wedge S(i)\geq (C',D')\wedge S(i)\geq S(v,p)\wedge S(i)=S(v,p),
    \end{displaymath}
    which, as $(C',D')\wedge S(i)$ has the same order as $S(v,p)$, implies that $(C',D')\wedge S(i)=S(v,p)$. Similarly $(D',C')\wedge S(i)=S(s,v)$.
\end{proof}
\begin{lem} \label{anchoreq}
Let $\Psi$ be a $k$-pseudoflower and $i$ a petal of $\Psi$ with predecessor $p$ and successor $s$. If a $k$-separation $(C,D)$ properly crossing $S(i)$ is anchored at $v \in C(I) \setminus I$, then $C \cap P_s = \emptyset$, $D \cap P_p = \emptyset$,$X \subseteq C \cap D$ and $(C \cap D) \setminus V(p,s) = P_v$.
\end{lem}
\begin{proof}
Since $P_s$ is contained in the separator of $S(i)$, but does not meet $V(v,p)$ it cannot meet $C$. Similarly $(C,D)\wedge S(i)=S(s,v)$ implies that $P_p$ does not meet $D$.
The third statement is immediate from the fact that $X$ occurs in the separators of both $S(s,v)$ and $S(v,p)$.

For the last equation, note that $(C,D)\wedge S(i)=S(v,p)$ implies both $C \cap V(s,p) = V(v,p)$ and $D \cup V(p,s) = V(p,v)$. Taking the intersection gives $(C \cap D \cap V(s,p)) \cup (C \cap V(s,p) \cap V(p,s))  = V(v,p) \cap V(p,v)$. Deleting $V(p,s)$ on both sides and simplifying gives the result.
\end{proof}

\begin{lem}\label{subdividepetalfinite}
    Let $\Phi$ be a $k$-pseudoflower distinguishing at least three elements of $\mathcal{P}$ located by $\Phi$.
    Let $(C,D)$ be a separation of order $k$ which properly crosses some petal separation $S(i)$ of $\Phi$ and is anchored at $v\in C(I)\setminus I$.
    Let $p$ be and $s$ be the predecessor and successor of $i$ in $C(I)$ respectively.
    Then there is an extension $\Phi'$ of $\Phi$, witnessed by $F$, such that
    \begin{itemize}
        \item For all $i'\in I\setminus \{i\}$ there is only one element contained in $F^{-1}(i')$.
        \item $F^{-1}(i)$ contains exactly three elements $i_1$, $m$ and $i_2$.
        By switching the names of $i_1$, $m$ and $i_2$ if necessary we can assume that $m$ is the successor of $i_1$ and the predecessor of $i_2$.
        \item $P_m=(C\cap D)\setminus V(s,p)$.
        \item The interval set of $i_1$ is $C\cap V(i)$ and the interval set of $i_2$ is $D\cap V(i)$.
        \item $(C,D)$ is an interval separation of $\Phi'$.
    \end{itemize}
\end{lem}
\begin{proof}

Let $\Phi'$ be obtained from $\Phi$ by replacing $i \in I$ with $i_1$, $m$ and $i_2$ in order and setting $P_m = (C \cap D) \setminus V(s,p)$, $P_{i_1} = C \cap P_i$ and $P_{i_2}= D \cap P_i$.

We will now prove that $\Phi'$ is a $k$-pseudoflower. 
The third condition is trivial.
By \cref{anchoreq} we know that $C \cap D$ consists of exactly $P_m$, $X$ and $P_v$, so $|P_m| = k-|P_v|-|X|=\frac{k-|X|}{2}$.

Assume for a contradiction that (\ref{star}) fails. We may assume without loss of generality that $|P_{i_1}|=\frac{k-|X|}{2}$. It follows that $P_m=P_{i_1}$ and thus $P_{i_1} \subseteq P_{i_2}$. But then $(D,C) \vee S(i)^*$ is cotrivial with witness $(C,D)$ and therefore is contained in every profile, contradicting the fact that $S(i)$ properly crosses $(C,D)$.

To show that $\Phi'$ is a $k$-pseudoflower, it thus remains to show the second condition.
Let $x,y \in C(I)$ be arbitrary. Because $\Phi$ was a $k$-pseudoflower, w.l.o.g. $y=m$. Furthermore, since $P_{i_1}$ and $P_{i_2}$ meet only in $P_m$, we have $V(x,m) \cap V(m,x)=P_x \cup P_m \cup X$.

Thus it is enough to show that $(V(x,m),V(m,x))$ is a separation. Without loss of generality $x$ occurs in the interval in the interval from $m$ to $v$ in $C(I)$, the other case is analogous. Note that, as $(C,D)$ is anchored at $v$,
$S(v,p)\leq (C,D)\leq S(v,s)$ and thus $V(v,p)\subseteq C\subseteq V(v,s)$.
So
\begin{displaymath}
V(x,p)\cup C = V(x,p)\cup (C\cap V(i)) = V(x,m)
\end{displaymath}
and similarly $V(p,x)\cap D=V(m,x)$.
Hence $S(v,p)\vee (C,D) = S(x,m)$, implying that $S(x,m)$ is indeed a separation.

Thus $\Phi'$ is a $k$-pseudoflower. By construction, $(C,D)$ appears as the interval separation $S(m,v)$. 
\end{proof}

\begin{figure}
    \centering
    \begin{tikzpicture}
\basicflower
\orientation

\labelofarcbelow[pictureorange]{130}{50}{1}{$i$}
\labelofarcabove[picturegreen]{130}{100}{1}{$i_1$}
\labelofarcabove[blue]{100}{50}{1}{$i_2$}
\newnode{$v$}{-110}
\newnode{$p$}{130}
\newnode{$m$}{100}
\newnode{$s$}{50}
\labelofsep{-110}{100}{$(C,D)$}{0}
\end{tikzpicture}
    \caption{The the index $i$ is replaced by the indices $i_1$ and $i_2$, creating a new cutvertex $m$. This allows us to display the separation $(C,D)$, see also \cref{subdividepetalfinite}.}
    \label{fig:subdividepetalinfinite}
\end{figure}
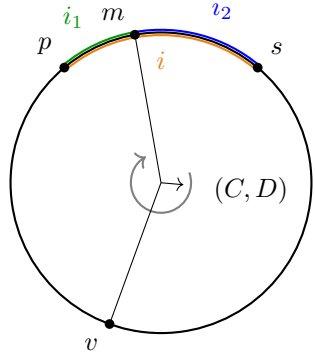

\section{\texorpdfstring{Maximal $k$-pseudoflowers}{Maximal pseudoflowers}}

Let $(\Phi_j)_{j\in J}$ be a $\leq$-chain of $k$-pseudoflowers which distinguish at least three elements of $\mathcal{P}$ that are located by the $\Phi_j$. In this section we will prove that there is a $k$-pseudoflower which is an upper bound of the chain $(\Phi_j)_{j\in J}$. The proof works for both $k$-pseudodaisies and $k$-pseudoanemones, but can be shortened a lot if the $\Phi_j$ are $k$-pseudoanemones as follows:
The careful construction of the sets $P_v$ with $v\in V_N$ in \cref{sec:addV_N} becomes redundant as all those sets are necessarily empty.
For the same reason, in \cref{sec:redundantpetals} there is no need to derive the sets $P_v$ with $v\in C(I\setminus I_N)\setminus (I\setminus I_N)$ from their counterparts in $C(I)\setminus I_N$.
Furthermore, it is possible to carry the following stronger version of (\ref{star}) through the limit process: That no petal should be empty.
This stronger version then makes all witnesses of comparability of $k$-pseudoflowers unique, making the search for a compatible choice of witnesses unnecessary.
Also, the definition of $I'$ can be made easier, as it simply consists of those $i\in I$ with $P_i\neq \emptyset$, and $C(I')$ is more obviously related to $C$.
Lastly, in $k$-anemones the condition of \cref{meetofintervalseps} that $P_a\cap P_d=\emptyset$ always holds, which means that the proof in \cref{sec:intervalsareseps} that all interval separations of the newly constructed $k$-pseudoflower are indeed interval separations can be streamlined.

\subsection{Finding compatible witnesses}

The first step in constructing an upper bound of a chain of $k$-pseudoflowers is to construct the index set.
For that we would like to take the inverse limit of the index sets, and for that we need compatible witnesses of the comparability of the $k$-pseudoflowers in the chain.
Phrased more formally, we need for every $j\leq l\in J$ a witness $F_{lj}$ of $\Phi_j\leq \Phi_l$ such that for $j\leq l\leq m\in J$ the concatenation of $F_{lj}$ with $F_{ml}$ is $F_{mj}$.

In \cref{concatenations} we showed that, given sufficiently large $k$-pseudoflowers $\Phi\leq \Phi'$ on index sets $I$ and $I'$, only few elements $i$ of $I'$ have several possible images under witnesses of $\Phi\leq \Phi'$, and that picking some possible image for every such $i$ gives again a witness of $\Phi\leq \Phi'$.
So defining the $F_{lj}$ amounts to finding a way of making all these choices in a way that is compatible over the whole chain of $k$-pseudoflowers.
For that, we first show how these choices interact in a chain of only three $k$-pseudoflowers.

So for the next two lemmas let $\Phi_1\leq \Phi_2\leq \Phi_3$ be $k$-pseudoflowers on index sets $I_1$, $I_2$ and $I_3$ respectively such that $\Phi_1$ extends a $k$-flower with three petals.
For some $i_3\in I_3$ we are now interested in possible images of $i_3$ under witnesses of $\Phi_2\leq \Phi_3$, their images under witnesses of $\Phi_1\leq \Phi_2$ and how those relate to witnesses of $i_3$ under witnesses of $\Phi_1\leq \Phi_3$.

\begin{lem}\label{witnessescompatible1}
If $i_3$ has two possible images $i_1$ and $i_1'$ under $\Phi_1\leq \Phi_3$, and there is $i_2\in I_2$ whose possible witnesses under $\Phi_1\leq \Phi_2$ are also $i_1$ and $i_1'$, then all witnesses of $\Phi_2\leq \Phi_3$ map $i_3$ to $i_2$.
\end{lem}
\begin{proof}
    By \cref{witnessexistspreimage} there is some $i_3'$ whose unique image under $\Phi_2\leq \Phi_3$ is $i_2$.
    Then $i_3'$ has $i_1$ and $i_1'$ as possible images under $\Phi_1\leq \Phi_3$ and $P_{i_3}=P_{i_3'}$ by \cref{nonuniqueimagesunderwitnesses}, so $i_3=i_3'$.
\end{proof}

\begin{lem}\label{witnessescompatible2}
If $i_3$ has two possible images $i_1$ and $i_1'$ under $\Phi_1\leq \Phi_3$, $i_1'$ successor of $i_1$ in $I_1$, and $i_3$ has two possible images $i_2$ and $i_2'$ under $\Phi_2\leq \Phi_3$, $i_2'$ successor of $i_2$ in $I_2$, then every witness of $\Phi_1\leq \Phi_2$ maps $i_2$ to $i_1$ and $i_2'$ to $i_1'$.
\end{lem}
\begin{figure}
    \centering
    \begin{tikzpicture}
    \orientation
    \begin{scope}[scale=1.15]
    \basicflower
    \labelofarcabove[red]{-5}{5}{1}{$i_3$}
    \labelofarcabove[blue]{30}{40}{1}{$j$}
    \labelofarcabove[red]{175}{185}{1}{$\hat{\imath}_3$}
    \end{scope}
    \begin{scope}[scale=0.8]
    \basicflower
    \labelofarcabove[pictureorange]{-15}{0}{1}{$i_2'$}
    \labelofarcabove[picturepurple]{0}{15}{1}{$i_2$}
    \labelofarcabove[picturegreen]{170}{190}{1}{$\hat{\imath}_2$}
    \end{scope}
    \begin{scope}[scale=0.45]
    \basicflower
    \labelofarcabove[pictureorange]{-20}{0}{1}{$i_1'$}
    \labelofarcabove[picturepurple]{0}{20}{1}{$i_1$}
    \labelofarcabove[red]{160}{200}{1}{$\hat{\imath}_1$}
    \end{scope}
    \end{tikzpicture}
    \caption{Three compatible $k$-pseudoflowers $\Phi_1\leq\Phi_2\leq \Phi_3$. The outer circle depicts $\Phi_3$ and three elements of its index set, the inner circle depicts $\Phi_1$.}
    \label{fig:witnessescompatible}
\end{figure}
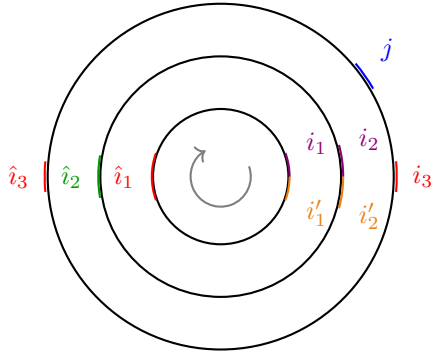
\begin{proof}
    See \cref{fig:witnessescompatible} for a depiction of the notation used in this proof.
    As a concatenation of witnesses is again a witness, all witnesses of $\Phi_1\leq \Phi_2$ map $i_2$ to one of $i_1$ and $i_1'$, and similarly for $i_2'$.
    If one element of $I_2$ had both $i_1$ and $i_1'$ as possible images under $\Phi_1\leq \Phi_2$, then by the previous lemma $i_3$ could not have both $i_2$ and $i_2'$ as possible images under $\Phi_2\leq \Phi_3$.
    Hence no element of $I_2$ has both $i_1$ and $i_1'$ as possible images under witnesses of $\Phi_1\leq \Phi_2$, and in particular both $i_2$ and $i_2'$ have a unique image under such witnesses.
    Let $\hat{\imath}_1$ be an element of $I_1$ that is neither $i_1$ nor $i_1'$, and let $\hat{\imath}_3$ be an element of $I_3$ that is mapped to $\hat{\imath}_1$ by all witnesses of $\Phi_1\leq \Phi_3$.
    
    Assume for a contradiction that both $i_2$ and $i_2'$ have the same image in $I_1$ under witnesses of $\Phi_1\leq\Phi_2$, and that this image is $i_1'$ (the other case is symmetric).
    Let $j$ be an element of $I_3$ that is mapped to $i_2$ by all witnesses of $\Phi_2\leq \Phi_3$.
    As no witness of $\Phi_2\leq \Phi_3$ maps $\hat{\imath}_3$ to $i_2$ or $i_2'$ and there is some witness that maps $i_3$ to $i_2'$, we have $j\in \mathopen]\hat{i_3},i_3\mathclose[$.
    Also, by \cref{modifywitnesses} (and because $j\neq i_3$) there is a witness of $\Phi_1\leq \Phi_3$ that maps $j$ to $i_1'$, $i_3$ to $i_1$ and $\hat{i_3}$ to $\hat{i_1}$.
    This implies $j\in \mathopen]i_3,\hat{\imath}_3\mathclose[$, a contradiction.
    
    So $i_2$ and $i_2'$ have distinct unique images under witnesses of $\Phi_1\leq \Phi_2$.
    If $\hat{\imath}_2$ denotes the image of $\hat{\imath}_3$ under some witness of $\Phi_2\leq \Phi_3$, then $\hat{\imath}_2\in \mathopen]i_2',i_2\mathclose[$ and $\hat{\imath}_1\in \mathopen]i_1',i_1\mathclose[$ implies that indeed all witnesses of $\Phi_1\leq \Phi_2$ map $i_2$ to $i_1$ and $i_2'$ to $i_1'$.
\end{proof}

Now we want to use these two lemmas to define the maps $F_{lj}$.
For all indices $j,l\in J$ with $j\leq l$ define a map $f_{lj}:I_l\rightarrow I_j$ as follows:
For $i\in I_l$, if there are indices $j=j_0<j_1<\ldots <j_n=l$, $n\geq 1$, in $J$ such that for $i_n=i$ and for every $m\leq n$ every witness of $\Phi_{j_{m-1}}\leq \Phi_{j_m}$ maps $i_m$ to the same element $i_{m-1}$ of $I_{m-1}$, then let $f_{lj}(i)=i_0$ (first case).
Otherwise there are two possible images $i_j$ and $i_j'$ of $i$ under witnesses of $\Phi_j\leq \Phi_l$ such that $i_j'$ is the successor of $i_j$.
In this case let $f_{lj}(i)=i_j'$ (second case).

\begin{lem}
	For all $j\leq l\in J$ and all $i\in I_j$, $f_{lj}(i)$ is well defined.
\end{lem}
\begin{proof}
    In order to show that $f_{lj}(i)$ is well-defined, it suffices to consider the case where $f_{lj}(i)$ is defined via the first case.
	Assume that $j=j_0\leq \ldots j_n=l$ in $J$ and $j=l_0\leq \ldots \leq l_m=l$ are two chains of elements of $J$ via which $f_{lj}(i)$ could be defined.
	We will show by induction on $n+m$ that the value of $f_{lj}(i)$ is the same in both cases.
	If one of $n$ and $m$ is $0$, then $j=l$ and the claim holds, so assume otherwise.
	Furthermore, if $j_{n-1}=l_{m-1}$, then it suffices to apply the induction hypothesis to the chains $j_0\leq \ldots \leq j_{n-1}$ and $l_0\leq \ldots \leq l_{m-1}$.
	So assume that $j_{n-1} < l_{m-1}$, the other case is symmetric.
	Let $r$ be an integer such that $l_{r-1}\leq j_{n-1}\leq l_r$, and $i_r$ the image of $i$ in $I_{l_r}$ under the chain of the $l_s$.
	If $i_r$ has several images under witnesses of $\Phi_{j_{n-1}}\leq \Phi_{l_r}$, then $i$ also has several images under witnesses of $\Phi_{j_{n-1}}\leq \Phi_l$.
	So $i_r$ has a unique image $i_s$ under witnesses of $\Phi_{j_{n-1}}\leq \Phi_{l_r}$, and that image is also the unique image of $i$ under witnesses of $\Phi_{j_{n-1}}\leq \Phi_l$.
	Then $j_0\leq \ldots \leq j_{n-1}\leq l_r$ and $l_0\leq \ldots \leq l_r$ are chains on which, because $n+r\leq n+m-1$, the induction hypothesis can be applied.
	Thus the lemma holds.
\end{proof}

\begin{lem}
	For all $j\leq l\leq m$ in $J$ and $i\in I_m$, $f_{lj}\circ f_{ml}(i)=f_{mj}(i)$.
\end{lem}
\begin{proof}
	First consider the case that $f_{mj}(i)$ is defined via the first case.
	Let $j_0\leq \ldots j_n$ be a sequence via which $f_{mj}(i)$ could have been defined, with $j_{r-1}\leq l\leq j_r$.
	If $i_r$ has two possible images $i_l,i_l'$ under witnesses of $\Phi_l\leq \Phi_{j_r}$, then $f_{ml}(i)$ is one of those.
	Also, every witness of $\Phi_{r-1}\leq \Phi_l$ maps both $i_l$ and $i_l'$ to $i_{r-1}$, implying that both $f_{lj}(i_l)$ and $f_{lj}(i_l')$ are defined via the first case and are equal to $f_{mj}(i)$.
	So in this case $f_{lj}\circ f_{ml}(i)=f_{mj}(i)$.
	If $i_r$ has a unique image $i_l$ under witnesses of $\Phi_j\leq \Phi_r$, then all witnesses of $\Phi_{r-1}\leq \Phi_l$ map $i_l$ to $i_{r-1}$, and thus $f_{ml}(i)$ is defined via the first case and also $f_{lj}(i_l)$ is defined via the first case.
	As $f_{mj}(i)$ is well defined, this implies that $f_{lj}\circ f_{ml}(i)=f_{mj}(i)$.
	
	So assume that $f_{mj}(i)$ is defined via the second case, and that $i_j$ and $i_j'$ are the two possible images of $i$ under witnesses of $\Phi_j\leq \Phi_m$.
	If there is $i_l\in I_l$ such that both $i_l$ and $i_l'$ are images of witnesses of $\Phi_j\leq \Phi_l$, then by \cref{witnessescompatible1} every witness of $\Phi_l\leq \Phi_m$ maps $i$ to $i_l$.
	Thus $f_{mj}(i)=f_{ml}(i_l)=f_{ml}\circ f_{lj}(i)$.
	So assume that there is no $i_l \in I_l$ that has $i_j$ and $i_j'$ as possible images under witnesses of $\Phi_j\leq \Phi_l$.
	By renaming assume that $i_j'$ is the successor of $i_j$.
	Then by \cref{witnessescompatible2} there are $i_l$ and $i_l'$ in $I_l$ such that every witness of $\Phi_j\leq \Phi_l$ maps $i_l$ to $i_j$ and $i_l'$ to $i_j'$.
	Also, $i_l'$ is the successor of $i_l$, and every witness of $\Phi_l\leq \Phi_m$ maps $i$ to $i_l$ or $i_l'$.
	Then $f_{ml}(i)$ cannot be defined via the first case, because that would imply that $f_{mj}(i)$ would also be defined via the first case.
	So $f_{mj}(i)=i_j'=f_{lj}(i_l')$ and $f_{ml}(i)=i_l'$.
\end{proof}

For $j\leq l\in J$ let $F_{lj}$ be the unique extension of $f_{lj}$ to a surjective map $C(I_l)\rightarrow C(I_j)$ respecting the cyclic order with $F_{lj}(I_l)=I_j$, which exists by \cref{cyclichoms1}.

\begin{lem}
    The maps $F_{lj}$ are witnesses of $\Phi_j\leq \Phi_l$ and are compatible in the sense that for $j\leq l\leq m\in J$ the concatenation of $F_{lj}$ and $F_{ml}$ is $F_{mj}$.
\end{lem}
\begin{proof}
As $f_{lj}$ can be obtained from the restriction of any witness of $\Phi_j\leq \Phi_l$ to $I_l$ by changing the images of some elements of $I_j$ to another image they can have under witnesses of $\Phi_j\leq \Phi_l$, by \cref{newwitness} $F_{lj}$ is also a witness of $\Phi_j\leq \Phi_l$.
Furthermore the restriction of $F_{lj}\circ F_{ml}$ to $I_m$ is equal to $f_{mj}$, and $F_{lj}\circ F_{ml}$ is surjective, respects the cyclic order and satisfies $F_{mj}(I_m)=I_j$.
Thus, by uniqueness of the extensions, $F_{mj}=F_{lj}\circ F_{ml}$.
\end{proof}

\begin{cor}
There are witnesses $F_{lj}$ of $\Phi_j\leq \Phi_l$, one for all pairs of indices $j\leq l\in J$, such that for all $j\leq l\leq m\in J$ the concatenation of $F_{lj}$ and $F_{ml}$ is $F_{mj}$.\qed
\end{cor}

\subsection{\texorpdfstring{Inverse limits of $k$-pseudoflowers}{Inverse limits of pseudoflowers}}\label{sec:limits}

As we showed in the last section, there is a compatible family of witnesses $F_{lj}$ of $\Phi_j\leq \Phi_l$, one for every pair of indices $j\leq l\in J$.
For the following construction we fix one such family.
Then there is an inverse limit $I$ of the $I_j$ with projections $(\pi_j)_{j\in J}$.
Because all $F_{l j}$ respect the cyclic order, the inverse limit also has a cyclic order which is respected by the projections. Further, because all the $F_{l j}$ are surjective, so are the projections.
We will construct a partition $\Psi:=(P_i)_{i\in C(I)}$ that is an upper bound of the chain $(\Phi_j)_{j\in J}$ and is a $k$-pseudoflower except that it may violate property (\ref{star}).
So far we can only define $P_i$ with $i\in I$: It is the intersection of all sets $P_{\pi_j(i)}$ with $j\in J$.

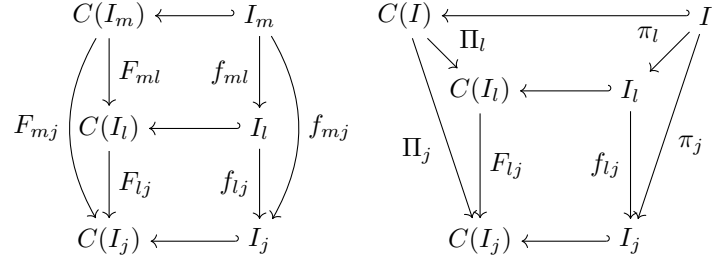
\begin{figure}
\centering
\begin{tikzpicture}
\draw node (A) at (1,0) {$I_j$};
\draw node (B) at (1,1.5) {$I_l$};
\draw node (C) at (1,3) {$I_m$};
\draw node (AM) at (-1,0) {$C(I_j)$};
\draw node (BM) at (-1,1.5) {$C(I_l)$};
\draw node (CM) at (-1,3) {$C(I_m)$};
\foreach \st/\ed/\l/\o in {B/A/f_{lj}/', C/B/f_{ml}/', C/A/f_{mj}/bend left, BM/AM/F_{lj}/, CM/BM/F_{ml}/}
\draw[->] (\st) to [\o, edge label = $\l$] (\ed);
\draw[->] (CM) to [bend right, edge label'=$F_{mj}$] (AM);
\foreach \st/\ed in {A/AM, B/BM, C/CM}
\draw[{Hooks[left]}->] (\st) to (\ed);
\end{tikzpicture}
\begin{tikzpicture}
\draw node (A) at (1,0) {$I_j$};
\draw node (B) at (1,2) {$I_l$};
\draw node (C) at (2,3) {$I$};
\draw node (AM) at (-1,0) {$C(I_j)$};
\draw node (BM) at (-1,2) {$C(I_l)$};
\draw node (CM) at (-2,3) {$C(I)$};
\foreach \st/\ed/\l/\s in {C/A/\pi_j/, B/A/f_{lj}/', C/B/\pi_l/', CM/AM/\Pi_j/', BM/AM/F_{lj}/, CM/BM/\Pi_l/near end}
\draw [->] (\st) to [\s,edge label = $\l$](\ed);
\foreach \st/\ed in {A/AM, B/BM, C/CM}
\draw [{Hooks[left]}->] (\st) -- (\ed);
\end{tikzpicture}
\caption{To the left: The maps $(F_{lj})_{j\leq l\in J}$ extend the compatible maps $(f_{lj})_{j\leq l\in J}$ and are therefore compatible witnesses that the $k$-pseudoflowers are comparable. To the right: The maps $(\Pi_j)_{j\in J}$ defined via the projections $(\pi_j)_{j\in J}$ are compatible with the maps $(F_{lj})_{j\leq l\in J}$.}
\label{fig:my_label}
\end{figure}

We will now relate $C(I)$ to the cyclically ordered sets $C(I_j)$.
By \cref{cyclichoms1} every projection $\pi_j:I\rightarrow I_j$ can be extended uniquely to a surjective map $\Pi_j:C(I)\rightarrow C(I_j)$ that respects the cyclic order.
For $j\leq l\in J$, the restriction of $F_{lj}\circ \Pi_l$ to $I_l$ is $\pi_j$, and hence $F_{lj}\circ \Pi_l=\Pi_j$.
By \cref{cyclichoms2}, the cutpoints of the $C(I_j)$ can be identfied with each other and with cutpoints of $C(I)$ as follows:
Given $j\leq l\leq n\in J$ and $v\in C(I_j)\setminus I_j$, let $u$ be the unique element of $C(I_l)\setminus I_l$ with 
$F_{lj}(u)=v$ and let $u'$ be the unique element of $C(I_n)\setminus I_n$ with $F_{nl}(u')=u$.
Then $u'$ is also the unique element of $C(I_n)\setminus I_n$ with $F(nj)(u')=v$.
Also, if $w$ is the unique element of $C(I)\setminus I$ with $\Pi_l(w)=u$, then $w$ is the unique element of $C(I) \setminus I$ with $\Pi_j(w)=v$.
So it is well-defined to identify $v$ with $w$ and with all cutpoints that get mapped to $v$ via some $F_{lj}$.
We do this identification for all cutpoints of all $C(I_j)$ and denote the set of cutpoints of $C(I)$ that are identified with a cutpoint of some $C(I_j)$ by $V_F$.
Then for $v\in V_F$, $P_v$ is the same for all $j\in J$ where $v$ is a cutpoint, so we take that vertex set also to be $P_v$ for $\Psi$.
Also note that, for $v,w\in V_F$, $V(v,w)$ does not depend on the $k$-pseudoflower $\Phi_j$ with respect to which it is defined.
But $V(v,w)$ taken in $\Psi$ is not yet defined, and when it is we will first have to show that it is equal to $V(v,w)$ taken in some $\Phi_j$.

In order to properly distinguish here, for cutpoints $v$ and $w$ of $C(I)$ we introduce the notation
\begin{align*}
    V'(v,w)&= X\cup P_v\cup P_w\cup\bigcup_{z\in [v,w]\cap (I\cup V_F)}P_z \\
    \intertext{and}\hat{V}(v,w)&=X \cup \bigcup_{z\in [v,w]\cap C(I)} P_z
\end{align*}
while the notation $V(v,w)$ is reserved for the case $v,w\in V_F$ and the value taken in some $\Phi_j$.
Note that up to now, in many cases $V'$ and $\hat{V}$ are not yet well-defined because $P_z$ is not yet defined for cutpoints $z\notin V_F$.
Those sets $P_z$ are going to be defined later.

For $v,w\in V_F$, we can see that $V'(v,w)$ and $V(v,w)$ are the same:

\begin{lem}\label{Vvwcompatible3}
For all distinct $v,w\in V_F$ we have $V(v,w)\setminus X=\bigcup_{z\in [v,w]}P_z$ where the interval is taken in $I\cup V_F$.
\end{lem}
\begin{proof}
Let us first show that $V(v,w) \setminus X$ is a subset of $\bigcup_{z\in [v,w]}P_z$. For this, let $u\in V(v,w)\setminus X$. If also $u\in V(w,v)$ then $u\in P_v\cup P_w$ and we are done, so assume that $u\notin V(w,v)$. If there is $z\in V_F$ such that $u\in P_z$, then $u\notin V(w,v)$ implies that $z\notin [w,v]$ and thus $z\in [v,w]$. If there is no $z\in V_F$ such that $u\in P_z$, then for all $j\in J$ there is a unique $i_j\in I_j$ such that $u\in P_{i_j}$. In this case also $F_{lj}(i_l)=i_j$ for all $j\leq l\in J$. So there is $i\in I$ such that $\Pi_j(i)=i_j$ for all $j\in J$, and $u\in P_i$. Again $u\notin V(v,w)$ implies $i\in [v,w]$.

To prove the other inclusion, let $j\in J$ be sufficiently large such that $C(I_j)$ contains $v$, $w$ and $z$ if $z\in V_F$. Then $\Phi_j$ witnesses that $P_z\subseteq V(v,w)$. Furthermore all $P_z$ are disjoint from $X$.
\end{proof}
\begin{cor}
$X=V(G)\setminus \bigcup_{z\in I\cup V_F}P_z$.\qed
\end{cor}
\begin{cor}
For all distinct $v,w\in V_F$ we have $V(v,w)=V'(v,w)$.\qed
\end{cor}
\begin{cor}\label{SofvandwforVF}
For all distinct $v,w\in V_F$ the pair $(V'(v,w),V'(w,v))$
is a separation of order at most $k$ with separator $P_v\cup P_w\cup X$.\qed
\end{cor}

\subsection{Completing the index set}\label{sec:addV_N}

Let $V_N$ be the set of cutpoints of $C(I)$ that are not contained in $V_F$.
In this subsection we will define the values $P_z$ with $z\in V_N$.

\begin{lem}\label{prop:vnotinPsiF}
Every $v\in V_N$ has a unique neighbor in $C(I)$, which is an element of $I$.
\end{lem}
\begin{proof}
As $v\in V_N$, no $\Pi_j$ maps $v$ to a cutpoint.
Let $i$ be the element of $I$ with $\pi_j(i)=\Pi_j(v)$ for all $j\in J$.
If it exists, let $z$ be a neighbor of $v$ in $C(I)$.
For all $j\in J$, $\Pi_j(z)$ is $\Pi_j(v)$ or one of its neighbors in $C(I_j)$.
Also no cutpoint can be a neighbor of another cutpoint by \cref{distinctvertices}, so $z\in I$ and $z=i$.
Hence if $v$ has a neighbor in $C(I)$ then that neighbor is $i$.
As $i$ is the only element of $I$ with $\Pi_j(i)=\Pi_j(v)$ for all $j\in J$, $i$ is indeed a neighbor of $v$ in $C(I)$.
\end{proof}

Now we want to define $P_v$ for $v\in V_N$. By \cref{prop:vnotinPsiF} there is a unique $i\in I$ which is a neighbor of $v$ in $C$. For every $j\in J$ let $u_j$ be the predecessor and $w_j$ the successor of $\Pi_j(i)$ in $C(I_j)$. Let $z\in V_F$ be a cutpoint such that for all sufficiently large $j\in J$ both $S(z,w_j)$ and $S(z,u_j)$ distinguish two elements of $\mathcal{P}$. In particular, for all sufficiently large indices $j$, $P_z$ is disjoint from $P_{u_j}$ and $P_{w_j}$ and thus $P_z$ is disjoint from $V(u_j,w_j)$.

\begin{figure}
    \centering
    \begin{tikzpicture}
\basicflower
\orientation
\labelofarcbelow[picturegreen]{130}{60}{1}{$\Pi_j(i)$}
\labelofarcabove[red]{100}{75}{1}{$i$}
\labelofsep[pictureorange]{100}{-130}{$(Y,Z)$}{1}
\labelofsep[blue]{230}{130}{$S(u_j,z)$}{1}
\newnode{$z$}{-130}
\newnode{$u_j$}{130}
\newnode{$v$}{100}
\newnode{$w_j$}{60}
\node at (\flowerradius +0.5cm,0) {};
\end{tikzpicture}
\begin{tikzpicture}
\basicflower
\orientation
\labelofarcbelow[picturegreen]{130}{60}{1}{$\Pi_j(i)$}
\labelofarcabove[red]{100}{75}{1}{$i$}
\labelofsep[blue]{100}{-50}{$(Y^1,Z^1)$}{1}
\labelofsep[pictureorange]{230}{100}{$(Z^2,Y^2)$}{1}
\labelofsep[red]{-50}{-130}{$S(z_2,z_1)$}{1}
\newnode{$z_2$}{-130}
\newnode{$z_1$}{-50}
\newnode{$u_j$}{130}
\newnode{$v$}{100}
\newnode{$w_j$}{60}
\node at (-\flowerradius - .5cm,0) {};
\end{tikzpicture}
    \caption{To the left: In the case that $v\in V_N$ is the predecessor of its neighbor $i$ in $C(I)$, the set $P_v$ is defined via the limit of the separations $S(z,u_j)$ for some suitable $z\in V_F$.\newline To the right: $P_v$ does not depend on the choice of $z$, see also \cref{Pvindofz}.}
    \label{fig:limitextended}
\end{figure}
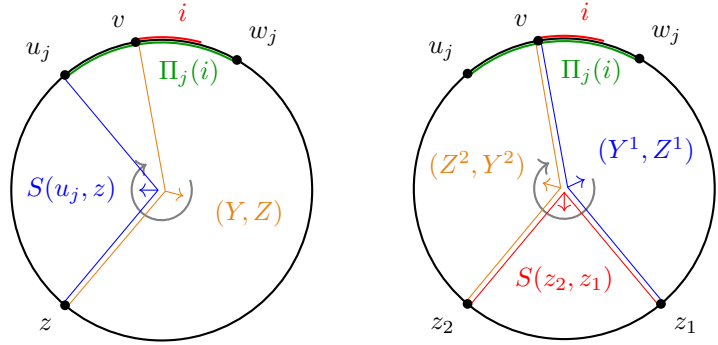

If $v$ is the predecessor of $i$, then let $(Y,Z)$ be the supremum of the set $\{S(z,u)\colon u\in \mathopen]z,v\mathclose[\cap V_F\}$ and otherwise let $(Y,Z)$ be the infimum of the set $\{S(z,w)\colon w\in \mathopen]v,z\mathclose[\cap V_F\}$.
In both cases $(Y,Z)$ has order at most $k$ because the order function is limit-closed.
Define $P_v:=(Y\cap Z)\setminus (P_z\cup X)$.
As $S(z,u_j) \leq (Y,Z)\leq S(z,w_j)$ for all sufficiently large $j\in J$, $Y\cap Z$ contains $P_z$.
Also $(Y,Z)$ distinguishes two elements of $\mathcal{P}$ so it has order $k$.
Hence $P_v$ has $(k-|X|)/2$ many elements.

\begin{lem}\label{placePv}
$P_v\subseteq V(u_j,w_j)$ for all $j\in J$.
\end{lem}
\begin{proof}
It suffices to show the claim for all $j\in J$ such that $z$ is a cutpoint of $C(I_j)$.
As $S(z,u_j)\leq (Y,Z)\leq S(z,w_j)$ by definition, $Y\cap Z$ is contained in $V(z,w_j)\cap V(u_j,z)$, which equals $V(u_j,w_j)\cup P_z$ by \cref{intersectionofintervalsets}.
As $P_v$ is a subset of $Y\cap Z$ and is disjoint from $P_z$, the lemma holds.
\end{proof}
\begin{cor}\label{cor:placePv}
$P_v\subseteq P_i$.
\end{cor}
\begin{proof}
As every $\Phi_j$ is a $k$-pseudoflower, for sufficiently large $j\in J$ we have
\begin{align*}
    V(u_j,w_j)=X \cup P_{\Pi_j(i)}\cup P_{u_j}\cup P_{w_j} = X \cup P_{\Pi_j(i)}.
\end{align*}
Because $P_v$ is disjoint from $X$, this implies $P_v\subseteq \bigcap_{j\in J}P_{\Pi_j(i)}=P_i$.
\end{proof}

\begin{cor}\label{cutpointcontainedinneighbor}
For $i'\in I$, if $u$ is a neighbor of $i'$ in $C(I)$ then $P_u\subseteq P_{i'}$.
\end{cor}
\begin{proof}
If $u\in V_N$, then $i'$ is the unique neighbor of $u$ and $P_u\subseteq P_{i'}$ by \cref{cor:placePv}.
If $u\in V_F$, then $P_u\subseteq P_{\Pi_j(i')}$ for all $j\in J$ for which $u$ is a cutpoint, also implying $P_u\subseteq P_{i'}$.
\end{proof}

\begin{lem}\label{Pvindofz}
The set $P_v$ does not depend on the choice of $z$.
\end{lem}
\begin{proof}
Let $z_1$ and $z_2$ be two possible choices for $z$, assume $z_1\in [v,z_2]$. Denote the sets defined by $z_n$ by a superscript index $n$, e.g.\ $(Y^1,Z^1)$ is the separation $(Y,Z)$ defined by $z_1$.
This proof is for the case that $v$ is the predecessor of $i$, the proof of the other case is symmetric.
Let $l$ be sufficiently large that both $S(z_1,w_l)$ and $S(z_2,u_l)$ are defined in $\Phi_l$ and distinguish two elements of $\mathcal{P}$.

As $S(z_1,u_j)\geq S(z_2,u_j)$ for all $j\geq l$, also $(Y^1,Z^1)\geq (Y^2,Z^2)$.
Furthermore for all $j\geq l$,
\begin{align*}
    S(z_1,u_l)\vee (Y^2,Z^2)\geq S(z_1,u_l) \vee S(z_2,u_j)=S(z_1,u_j).
\end{align*}
So $(Y^1,Z^1)\leq S(z_1,u_l)\vee (Y^2,Z^2)\leq (Y^1,Z^1)$ and thus the two separations are equal.
In particular $P_v^1\subseteq Z^1\subseteq Z^2$ and $P_v^1\subseteq Y^2\cup V(z_1,u_l)$.
Because also $P_v^1\subseteq V(u_l,w_l)$, together we have
\begin{align*}
    P_v^1\subseteq V(u_l,w_l)\cap (Y^2\cup V(z_1,u_l))=(V(u_l,w_l)\cap Y^2) \cup (X\cup P_{u_l})\subseteq Y^2
\end{align*}
where the equality holds by \cref{intersectionofintervalsets}.
Hence $P_v^1\subseteq Y^2\cap Z^2$.
Also $P_v^1\subseteq V(u_j,w_j)$ for all sufficiently large $j\in J$ implies that $P_v^1\cap V(z_1,z_2)=\emptyset$ and thus $P_v^1=P_v^2$ as those two sets have the same size.
\end{proof}

\begin{lem}\label{Vvzcompatible1}
$(Y,Z)=(V'(z,v), V'(v,z))$.
\end{lem}
\begin{proof}
We are going to show the lemma in the case that $v$ is the predecessor of $i$, the other case is symmetric.

We have $V'(z,v)=P_v\cup \bigcup_{t\in [z,v]\cap V_F}V(z,t)\subseteq Y$. Assume there is $u\in Y\setminus V'(z,v)$, then $u\notin V(z,u_j)$ for all sufficiently large $j\in J$. So $u\in Z$ and thus $u\in P_z\cup P_v\cup X$, a contradiction. So $Y=V'(v,z)$.

We have $V'(v,z)\subseteq V'(u_j,z)=V(u_j,z)$ for all sufficiently large $j\in J$ and thus $V'(v,z)\subseteq Z$. Let $u\in Z\setminus (X\cup P_v)$ and let $w\in I\cup V_F$ be some element such that $u\in P_w$. If $w\in [v,z]$, then also $u\in V'(v,z)$. Otherwise $w\in [z,v]$, so $w\in [z,u_j]$ for some $j\in J$ and thus $u\in V(z,u_j)\subseteq Y$. In this case we have $u\in Y\cap Z= P_v\cup P_z\cup X\subseteq V'(v,z)$. Thus $Z\subseteq V'(v,z)$, and these two sets are equal.
\end{proof}
\begin{cor}\label{Sofvwforzandv}
$(V'(z,v)\cup P_v,V'(v,z)\cup P_v)$ is a separation with separator $P_v\cup P_z\cup X$.
\end{cor}

\subsection{\texorpdfstring{Interval separations are indeed $k$-separations}{Interval separations are indeed separations}}\label{sec:intervalsareseps}
The goal of this section is to show that the interval separations of the limit we are constructing have the properties required in the definition of pseudoflowers.
\begin{lem}\label{VprimetoV}
For all distinct $v,w\in V_F\cup V_N$ we have $V'(v,w)=\hat{V}(v,w)$.
\end{lem}
\begin{proof}
$V'(v,w)\subseteq \hat{V}(v,w)$ is clear by definition of $V'(v,w)$.
In order to show the reverse inclusion, let $t$ be an element of $[v,w]$. and $u\in P_t$.
If $t\in \{v,w\}$ or $t\in I\cup V_F$, then clearly $P_t\subseteq V'(v,w)\cup P_v\cup P_w$.
Otherwise $t$ has a neighbor $i\in I$, and $i\in [v,w]$.
Then by \cref{cor:placePv}, $P_v\subseteq P_i\subseteq V'(v,w)$.
\end{proof}

\begin{lem}\label{lemmaforV}
(The cyclic order is illustrated in \cref{fig:intersectionofintervalsets}).
Let $a$, $b$, $c$ and $d$ be elements of $V_F\cup V_N$ such that
\begin{itemize}
    \item $b\in [a,c]$ and $d\in [c,a]$
    \item $(V'(a,c),V'(c,a))$ is a separation with separator $X\cup P_a\cup P_c$
    \item $(V'(b,d),V'(d,b))$ is a separation with separator $P_b\cup P_d\cup X$.
    \item $P_a\cap P_d=\emptyset$.
\end{itemize}
Then $(V'(b,c),V'(c,b))$ is a separation with separator $P_b\cup P_c\cup X$.
\end{lem}
\begin{proof}
We show that
\begin{align*}
    (V'(b,c),V'(c,b))=(V'(a,c),V'(c,a)) \wedge (V'(b,d),V'(d,b)).
\end{align*}
$V'(c,b)=V'(c,a)\cup V'(d,b)$ and $V'(b,c)\subseteq V'(a,c)\cap V'(b,d)$ clearly hold.
Let $u\in V'(c,b)\setminus (P_b\cup P_c\cup X)$.
Because $P_a\cap P_d$ is empty, either $u\in V'(c,a)\setminus (P_a\cup P_c\cup X)$ or $u\in V'(d,b)\setminus (P_b\cup P_d\cup X)$.
In the first case, because $(V'(a,c),V'(c,a))$ is a separation with separator $P_a\cup P_c\cup X$, $u$ is an element of $V'(c,a)\setminus V'(a,c)$.
In the second case we similarly get $u\in V'(d,b)\setminus V'(b,d)$.
Hence in both cases $u\notin V'(a,c)\cap V'(b,d)$, showing that $V'(a,c)\cap V'(b,d)=V'(b,c)$.
Furthermore, $V'(b,c)\cap V'(c,b)\subseteq P_b\cup P_c \cup X$, and so the inclusion must be an equality.
\end{proof}

\begin{lem}\label{Phiispseudoflower}
For all distinct $v,w\in V_F\cup V_N$ the pair $(\hat{V}(v,w),\hat{V}(w,v))$ is a separation with separator $X\cup P_v\cup P_w$.
\end{lem}
\begin{proof}
By \cref{VprimetoV} it suffices to show that $(V'(v,w),V'(w,v))$ is a separation with the correct separator for all distinct elements $v$ and $w$ of $V_F\cup V_N$.
If both $v$ and $w$ are contained in $V_F$ then this is true by \cref{SofvandwforVF}. Consider first the case that exactly one is contained in $V_F$, by switching the names we may assume $w\in V_F$.
If $w$ is a suitable candidate for $z$ in the definition of $P_v$, then $(V'(w,v),V'(v,w))$ is a separation with separator $P_v\cup P_w\cup X$ by \cref{Sofvwforzandv}.
So assume otherwise.
We will consider the case that for all sufficiently large $j\in J$, $S(w,u_j)$ does not distinguish elements of $\mathcal{P}$, the case where $S(w,w_j)$ does not distinguish elements of $\mathcal{P}$ is symmetric.

Let $z\in V_F$ from which $P_v$ might have been defined.
Then by choice of $z$ there is $t\in V_F$ such that $S(t,z)$ distinguishes elements of $\mathcal{P}$ and such that $t\in [v,z]$ if $w\in [z,v]$ and $t\in [z,v]$ if $w\in [v,z]$.
Also $(V'(z,v),V'(v,z))$ is a separation with separator $P_z\cup P_v\cup X$ by \cref{Sofvwforzandv} and we already saw that $(V'(w,t),V'(t,w))$ is a separation with separator $P_w\cup P_t\cup X$.
Furthermore $S(t,z)=(V'(t,z),V'(z,t))$, and because $S(t,z)$ distinguishes elements of $\mathcal{P}$ and thus has connectivity $k$ this implies $P_t\cap P_z=\emptyset$.
Now we can apply \cref{lemmaforV}.
If $t\in [v,z]$ then we apply \cref{lemmaforV} for $a=z$, $b=w$, $c=v$ and $d=t$.
Otherwise we apply the lemma for $a=t$, $b=v$, $c=w$ and $d=z$.
In both cases $(V'(w,v),V'(v,w))$ is a separation with separator $P_v\cup P_w\cup X$.

Now assume that both $v$ and $w$ are contained in $V_N$. Because every $\Phi_j$ distinguishes at least three elements of $\mathcal{P}$, by swapping the names of $v$ and $w$ if necessary we may assume that there are $t$ and $u$ in $V_F$ such that $t\in [v,w]$ and $u\in [t,w]$ and $S(t,u)$ distinguishes elements of $\mathcal{P}$.
We already showed that $(V'(v,u),V'(u,v))$ and $(V'(t,w),V'(w,t))$ are $k$-separations with separator $P_v\cup P_u\cup X$ and $P_t\cup P_w\cup X$ respectively. Furthermore $P_t\cap P_u=\emptyset$ because $S(t,u)$ distinguishes elements of $\mathcal{P}$, thus we may apply \cref{lemmaforV} to $a=u$, $b=w$, $c=v$ and $t=d$ and we are done.
\end{proof}

\subsection{Deletion of redundant petals}\label{sec:redundantpetals}

$\Phi$ is now nearly a $k$-pseudoflower, the only property missing is (\ref{star}). In order to fix that, we are going to delete troublesome elements from $I$.

Let $\mathcal{Y}$ be the set of subsets of $V$ of size $(k-|X|)/2$ that are of the form $P_i$ for some $i\in I$.
In order to ensure that (\ref{star}) holds, it suffices to delete, for every $Y\in \mathcal{Y}$, all but one index $i\in I$ with $P_i=Y$.
For that, we first define one element we want to keep:
If there is $j\in J$ and $i_j\in I_j$ whose petal is $Y$, then we want to define $i_Y$ to be an element of $I$ with $\pi_j(i_Y)=i_j$.
If that does not exist, then we pick $i_Y$ arbitrarily with $P_{i_Y}=Y$.

Actually if $i_Y$ is defined via the first case, then $i_Y$ is uniquely determined:

\begin{lem}
Let $Y\in \mathcal{Y}$ and $i_j$ be an index element of some $\Phi_j$ whose petal is $Y$.
Then there is a unique $i\in I$ with $\pi_j(i)=i_j$, and if $P_{i_l}=Y$ for some $l\in J$ and $i_l\in I_l$, then $\pi_l(i)=i_l$.
\end{lem}
\begin{proof}
Let $l\in J$ with $j\leq l$.
By \cref{witnessexistspreimage} there is some $i_l\in I_l$ that is mapped to $i_j$ by all witnesses of $\Phi_j\leq \Phi_l$, in particular $F_{lj}(i_l)=i_j$.
Thus $P_{i_l}$ is contained in $P_{i_j}$ and has size at least $(k-|X|)/2$.
But $P_{i_j}$ also has size $(k-|X|)/2$, so $P_{i_l}=P_{i_j}=Y$.
By (\ref{star}), $i_l$ is the only element of $I_l$ with $P_{i_l}=Y$, and thus $i_l$ is the only element of $I_l$ with $F_{lj}(i_l)=i_j$.

So if $i$ and $i'$ are elements of $I$ with $\pi_j(i)=\pi_j(i')=i_j$, then $\pi_l(i)=\pi_l(i')$ for all $l\geq j$, implying that $i=i'$.
By construction of $I$ there is some $i\in I$ with $\pi_j(i)=i_j$, thus there is a unique such element.
\end{proof}

The previous lemma implies that if there is some index $i_j$ of some $\Phi_j$ with $P_{i_j}=Y$, then there is $i\in I$ such that all indices of any $\Phi_j$ whose petal is $Y$ are the image of $i$ under $\pi_j$.

Obtain $I'$ from $I$ by deleting, for every element $Y$ of $\mathcal{Y}$, all $i\in I-i_Y$ with $P_i=Y$.
Obtain $C$ from $C(I)$ by deleting $I\setminus I'$ and, for all $i\in I\setminus I'$, the predecessor of $i$ in $C(I)$.

The following lemma implies that the elements of $I'$ are close to being ``dense'' in $C(I)$:

\begin{lem}\label{nondeleteddense}
Let $w$ and $w'$ be distinct elements of $V_F$ such that for some $Y\in \mathcal{Y}$, $P_z=Y$ for all $z\in [w,w']\setminus I$.
Then for some $i'\in [w,w']\cap I$ there is $j\in J$ with $P_{\pi_j(i')}=Y$ or there is $y\in P_{i'}$ such that $y\notin P_{i''}$ for all $i''\in I-i'$.
\end{lem}
\begin{proof}
Let $j\in J$ be large enough that $C(I_j)$ contains both $w$ and $w'$.
Then there is $i_j\in [w,w']\cap I_j$, where the interval is taken in $C(I_j)$, and some $i'\in I$ with $\pi_j(i')=i_j$.
For the predecessor $p$ and the successor $s$ of $i_j$ in $C(I_j)$, $P_s=P_p=Y$.
If $P_{i_j}=Y$, then the lemma holds, so assume otherwise.
Let $y$ be an element of $P_{i_j}$ that is not contained in $Y$.
Then $s$ and $p$ are contained in $[w,w']$ and $V(p,s)$ does not depend on whether it is taken in $(P_z)_{z\in C(I)}$ or in $\Phi_j$.
Hence there is $z\in [w,w']$ such that $y\in P_z$, and $z$ has to be contained in $I$.
So $y\in P_z$ for some $z\in [w,w']\cap I$, implying that the predecessor $p'$ and successor $s'$ of $z$ in $C(I)$ both have $Y$ as vertex set.
Then the separating set of the separation $S(z)$ is $Y$, implying that $y$ is only contained in $P_z$ and in no $P_{i''}$ with $i''\in I-z$.
\end{proof}

So every element of $C(I)\setminus C$ is close to some element of $C$:

\begin{lem}\label{emulatedeleted}
Let $z\in C(I)\setminus C$.
Then there is $v\in C\setminus I'$ such that $[z,v]$ is finite, no element of $[z,v\mathclose[$ is contained in $C$, and $P_{z'}=P_v$ for all $z'\in [z,v\mathclose[$.
\end{lem}
\begin{proof}
If $z\in I$, then let $i_0=z$, otherwise let $i_0$ be the successor of $z$.
Recursively, if $i_l$ is defined and has a successor in $I$ that is not contained in $I'$, then define $i_{l+1}$ to be that successor.

Assume for a contradiction that $i_2$ is defined.
For any distinct $i,i'\in I$, some projection $\pi_j$ maps $i$ and $i'$ to different images.
So there is $w\in V_F$ such that $w\in [i,i']$.
Thus the common neighbor $w_0$ of $i_0$ and $i_1$, and the common neighbor $w_1$ of $i_1$ and $i_2$, are both contained in $V_F$.
Now the previous lemma implies that $i_1$ is contained in $I'$, a contradiction.

If $i_1$ is defined, then let $v$ be the successor of $i_1$ in $C(I)$.
Otherwise let $v$ be the successor of $i_0$ in $C(I)$.
Thus, if $v$ is the predecessor of some element of $I$, then that is also contained in $I'$.
For $i\in I\setminus I'$, and the predecessor $p$ and successor $s$ of $i$ in $C(I)$, $P_p=P_i=P_s$.
In particular $P_z=P_{i_0}=P_v$, and if $i_1$ is defined then also $P_v=P_{i_1}=P_{w_0}$ for the common neighbor $w_0$ of $i_0$ and $i_1$ in $C(I)$.
\end{proof}

\begin{lem}
$C$ is a cycle completion of $I'$.
\end{lem}
\begin{proof}
In order to show that the identity on $I'$ can be extended uniquely to a bijective monotone map from $C$ to $C(I')$ it suffices to show that all non-trivial intervals of $I'$ can be uniquely written as $[v,w]\cap I'$ for elements $v$ and $w$ of $C\setminus I'$.
In order to do so, let $I''$ be a non-trivial interval of $I'$, and let $i\in I'\setminus I''$.
Let $\hat{I}$ be the set of all $\hat{\imath}$ for which there are $i_1,i_2\in I''$ such that the interval $[i_1,i_2]$, taken in $I$, contains $\hat{\imath}$ but not $i$.
Then $\hat{I}$ is an interval of $I$ that contains $I''$ as a subset and that does not contain $i$.
So $\hat{I}$ can be written as $[v,w]\cap I$ for unique elements $v$ and $w$ of $C(I)\setminus I$.
Also, $I''=\hat{I}\cap I'=[v,w]\cap I'$.

In order to find $v$ and $w$ that are contained in $C$, first consider $v$.
If $v$ is the predecessor of some $i'\in I$, then $i'$ is contained in $I''$ and thus in $I'$.
So $v$ cannot be the predecessor of an element of $I\setminus I'$, implying that $v$ is contained in $C$.
But $w$ could indeed be contained in $C(I)\setminus C$.
If it is, then apply \cref{emulatedeleted} to $w$ and obtain $w'$.
Otherwise let $w=w'$.
Then $\hat{I}=[v,w]\cap I'=[v,w']\cap I'$.

In order to show that $v$ and $w$ are unique, it suffices to show for any distinct $v,w\in C\setminus I'$ that $[v,w]\cap I'$ is non-empty.
As in the proof of \cref{emulatedeleted}, if $[v,w]$ contains enough elements (at least 7) then it contains two elements of $V_F$ and thus an element of $I'$.
Otherwise, $[v,w]$ is finite and $v$ is the predecessor in $C(I)$ of some $i\in [v,w]\cap I$.
Because $v\in C$, this implies that $i\in I'$.
\end{proof}

Let $\tilde{F}:C(I')\rightarrow C$ be a bijective monotone map whose restriction to $I'$ is the identity.
This map then identifies $C(I')$ with $C$.
Denote the family $(P_z)_{z\in C(I')}$ where each $P_z$ is equal to $P_{\tilde{F}(z)}$, by $\Phi'$.
As $\Phi$ satisfies all properties of a $k$-pseudoflower with the possible exception of (\ref{star}), also $\Phi'$ satisfies all properties of a $k$-pseudoflower with the possible exception of (\ref{star}).

\begin{lem}\label{Vwcompatible4}
For all elements $v$ and $w$ of $C\setminus I'$ we have
\begin{displaymath}
\bigcup_{z\in [v,w]_C}P_z=\bigcup_{z\in [v,w]_{C(I)}}P_z
\end{displaymath}
where the intervals are taken in $C$ and $C(I)$, respectively.
\end{lem}
\begin{proof}
It is clear that the left-hand side is a subset of the right-hand side.
In order to show the other direction, let $u\in [v,w]_{C(I)}$.
If $u\in C$ then $u\in [v,w]_C$, so assume otherwise.
Apply \cref{emulatedeleted} to $u$ and obtain $z'\in C\setminus I'$.
Then $w\notin [u,z'\mathclose[$, as $w\in C$, so $z'\in [v,w]$.
As $P_u=P_{z'}$, the lemma holds.
\end{proof}

\begin{lem}\label{Phiprimeispseudoflower}
$\Phi'$ is a $k$-pseudoflower such that $\Phi_j\leq \Phi'$ for all $j\in J$.
\end{lem}
\begin{proof}
Recall that by $\tilde{F}$, $C(I')$ is identified with $C$.
For all $v,w\in C(I')\setminus I'$,
\begin{flalign*}
    &&\bigcup_{z\in [v,w]_{C(I')}}P_z&=\bigcup_{z\in [v,w]
    _C}P_z&&\text{by \cref{Vwcompatible4}}\\
    &&&=V'(v,w)\setminus X&&\text{by \cref{VprimetoV}}
\end{flalign*}
so by \cref{Phiispseudoflower} $S(v,w)$ taken in $\Phi'$ is a separation with separator $P_v\cup P_w\cup X$ and also $X=V\setminus \bigcup_{z\in C(I')}P_z$.
As every $P_v$ with $v\in C(I)\setminus I$ has size $(k-|X|)/2$, this is also true for every $v\in C(I')\setminus I'$.
Hence every $S(v,w)$ has order at most $k$.
If some $i\in I'$ has a neighbor $v$ in $C(I')$, then $P_v\subseteq P_i$ by \cref{cutpointcontainedinneighbor}.
And by definition of $I'$, $\Phi'$ satisfies (\ref{star}).
Thus $\Phi'$ is a $k$-pseudoflower.

In order to show that $\Phi_j\leq \Phi'$ for all $j\in J$, we first show that the restriction of $\pi_j$ to $I'$ is surjective.
For that, let $i_j\in I_j$ and let $p$ and $s$ be the predecessor and successor of $i_j$ in $C(I_j)$.
Then $p$ and $s$ are contained in $V_F$, and thus by \cref{nondeleteddense} there is $i\in [v,w]\cap I'$, and $\pi_j(i)=i_j$.
So the restriction of $\pi_j$ to $I'$ with codomain $I_j$ is a surjective monotone map, and there is a unique monotone surjective extension $F_j:C(I')\rightarrow C(I_j)$ with $F_j(I')=(I_j)$.

Let $v$ and $w$ be distinct elements of $C(I_j)\setminus I_j$.
If $v\in C$, then let $v'=v$.
Otherwise there is, by \cref{emulatedeleted}, an element $v'\in C\setminus I'$ such that $[v,v']\cap C = v'$ and $P_v=P_{v'}$.
Define $w'$ similarly.
Then
\begin{displaymath}
V(v,w)=\hat{V}(v,w)=\hat{V}(v',w')=V(\tilde{F}^{-1}(v'), \tilde{F}^{-1}(w'))
\end{displaymath}
so $F_j$ witnesses that $\Phi_j\leq \Phi'$.
\end{proof}

So in this section it was shown that if $(\Phi_j)_{j\in J}$ is a $\leq$-chain of $k$-pseudoflowers which distinguish at least three elements of $\mathcal{P}$ that they locate, then there is a $k$-pseudoflower which is an upper bound of the chain $(\Phi_j)_{j\in J}$.
In particular, if $\Psi$ is a $k$-pseudoflower which distinguishes at least three elements of $\mathcal{P}$ that it locates, then in the set of $k$-pseudoflowers $\Phi'$ with $\Psi\leq \Phi'$ every $\leq$-chain has an upper bound.
Thus the following theorem follows by Zorn's Lemma:
\begin{thm}\label{leqmaxexists}
    Let $\Phi$ be a $k$-pseudoflower which distinguishes at least three elements of $\mathcal{P}$ that it locates.
    Then there is a $\leq$-maximal $k$-pseudoflower $\Psi$ such that $\Phi\leq \Psi$.
\end{thm}

In this theorem, the condition that $\Phi$ distinguishes at least three elements of $\mathcal{P}$ that it locates can be weakened to $\Phi$ distinguishing any three profiles or tangles that it locates.
As every $k$-pseudoflower locates every $k$-tangle, and every $k$-pseudoflower extending a $k$-flower with four petals locates every $k$-profile, the theorem can be specialized to the following versions:

\begin{thm}\label{leqmaxexistsprofile}
    Let $\Phi$ be a $k$-pseudoflower which distinguishes at least three  $k$-profiles and extends a $k$-flower with four petals.
    Then there is a $\leq$-maximal $k$-pseudoflower $\Psi$ such that $\Phi\leq \Psi$.
\end{thm}

\begin{thm}\label{leqmaxexiststangle}
    Let $\Phi$ be a $k$-pseudoflower which distinguishes at least three $k$-tangles.
    Then there is a $\leq$-maximal $k$-pseudoflower $\Psi$ such that $\Phi\leq \Psi$.
\end{thm}

\subsection{\texorpdfstring{Existence of $\preccurlyeq$-maximal $k$-pseudoflowers}{Existence of maximal pseudoflowers}}\label{sec:preccurlymax}

Call an element $P$ of $\mathcal{P}$ \emph{closed} if whenever $(S_j)_{j\in J}$ is a chain of elements of $P$, then its supremum (which exists because their order is bounded) is contained in $\mathcal{P}$.
If $\mathcal{P}$ only contains closed elements, then $\leq$-maximal $k$-pseudodaisies are also $\preccurlyeq$-maximal, and thus there are $\preccurlyeq$-maximal $k$-pseudodaisies.

\begin{figure}
\centering
\begin{tikzpicture}
\basicflower
\orientation
\newnode{$v$}{90}
\newnode{$t$}{0}
\newnode{$w$}{-90}
\newnode{$y$}{180}
\labelofsep{90}{0}{$S(v,t)$\\$\in P_3$}{1}
\labelofsep{0}{-90}{$S(t,w)$\\$\in P_4$}{1}
\labelofarcbelow[red]{270}{90}{1}{$W$}
\end{tikzpicture}
\caption{Some of the notation used in the proof of \cref{maximalflowerwrttangles}.}
\label{fig:maximalflowerwrttangles}
\end{figure}
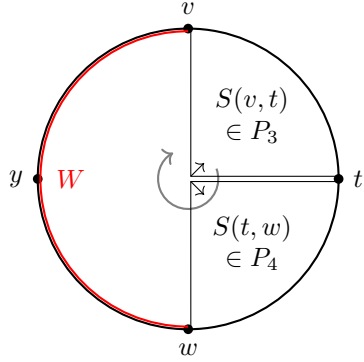

\begin{thm}\label{maximalflowerwrttangles}
Let $\Phi$ be a $k$-pseudoflower that is $\leq$-maximal and distinguishes three elements of $\mathcal{P}$ that it locates.
If all elements of $\mathcal{P}$ are closed, then $\Phi$ is also $\preccurlyeq$-maximal.
\end{thm}
\begin{proof}
See \cref{fig:maximalflowerwrttangles} for a depiction of some of the notation.
Let $\Phi$ be a $k$-pseudoflower that distinguishes three elements of $\mathcal{P}$ that it locates and which is not $\preccurlyeq$-maximal among all $k$-pseudoflowers, as witnessed by $\Phi'$. Let $P_1$ and $P_2$ be two elements of $\mathcal{P}$ which are distinguished by $\Phi'$ but not by $\Phi$.

By \cref{placetanglesinflower} there is a cutpoint $v\in C(I)\setminus I$ such that either $S(v,w)\in P_1$ for all $w\in C(I)\setminus I-v$ or $S(w,v)\in P_1$ for all $w\in C(I)\setminus I-v$.
We will assume that $S(v,w)\in P_1$ for all $w\in C(I)\setminus I-v$, the other case is symmetric.
Because $\Phi$ distinguishes at least three elements of $\mathcal{P}$ there is an interval separation $S(v,w)$ of $\Phi$ whose inverse is contained in two elements $P_3$ and $P_4$ of $\mathcal{P}$ which are distinguished and located by $\Phi$.
As $\Phi$ distinguishes and locates $P_3$ and $P_4$, and distinguishes them from $P_1$, there is $t\in \mathopen]v,w\mathclose[$ such that $S(v,t)$ distinguishes $P_3$ and $P_4$.
By swapping the names of $P_3$ and $P_4$ if necessary, we may assume that $S(t,v)\in P_3$.
Again by the profile property also $S(t,w)$ distinguishes $P_3$ and $P_4$ with $S(w,t)\in P_4$.

The $k$-pseudoflower $\Phi'$ distinguishes $P_1$, $P_2$, $P_3$ and $P_4$ pairwise, so there is an interval separation $(C,D)$ of $\Phi'$ which is contained in $P_3$ but not $P_4$ and which distinguishes $P_1$ and $P_2$.
By swapping the names of $P_1$ and $P_2$ if necessary we may assume that $(C,D)$ is contained in $P_1$ and $P_3$ and that its inverse is contained in $P_2$ and $P_4$.

If $v$ has a predecessor $v'$ in $C(I)\setminus I$, then $(C,D)$ properly crosses $S(v,v')$ and by \cref{findanchoringset,subdividepetalfinite} there is an extension $\Phi''$ of $\Phi$ which distinguishes $P_1$ and $P_2$.
Assume for a contradiction that $v$ has no predecessor in $C(I)\setminus I$.
Let $W$ be the interval $\mathopen]w,v\mathclose[\setminus I$.
If $\Phi$ is a $k$-pseudoanemone, then $S(t,v)$ is the supremum of the separations $S(t,x)$ with $x\in W$.
Because $P_1$ is closed and contains all $S(t,x)$ with $x\in W$ but not $S(t,v)$, $\Phi$ cannot be a $k$-pseudoanemone and thus has to be a $k$-pseudodaisy.

For all $x\in W$ let $S_x$ be the separation $((C,D)\vee S(t,x))\wedge S(t,v)$.
All three $k$-separations $(C,D)$, $S(t,x)$ and $S(t,v)$ are contained in $P_3$, not contained in $P_4$ and have order at most $k$, so $S_x$ has order at most $k$.
Denote the unique supremum of $(S_x)_{x\in W}$, which exists by limit-closedness, by $(A,B)$.
As the union of all sets $V(v,x)$ with $x\in W$ is the whole ground set $V$, there is some $y\in W$ such that $V(v,y)$ contains $A\cap B$.
But all profiles are closed, so $(A,B)$ is still contained in $P_1$ and thus properly crosses $S(v,y)$.
Hence applying \cref{anchoreq} to the concatenation of $\Phi$ on vertices $v$, $t$ and $y$ shows that $A\cap B$ contains a vertex not in $V(v,y)$, a contradiction to the choice of $y$.
\end{proof}

\section{\texorpdfstring{Example of an infinite $k$-pseudoflower}{Example of an infinite pseudoflower}}\label{sec:example}

Recall that the relevant definitions for graph-like spaces can be found for example in \cite{BCC:graphic_matroids}.

\begin{ex}\label{ex:infinitedaisy}
Let $a$ and $n$ be natural numbers and define $k=a+2n$. Let $G$ be a $k$-connected graph and $(P_i)_{1\leq i\leq n}$ a family of pairwise disjoint arcs. Let $V$ be a vertex set of $G$ 
\begin{figure}
\centering
\includegraphics{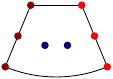}
\includegraphics{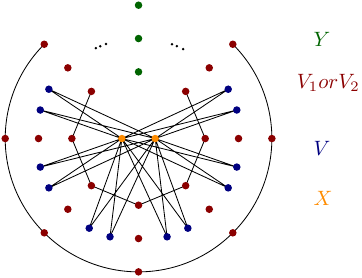} 
\caption{An example of a $k$-pseudoflower with $|X|=2$ and $d=3$.}
\label{examplesofflowers}
\end{figure}
of size $a$ which does not contain any first or last vertices of some $P_i$. Then we can define a $k$-connected graph-like space $G'$ as follows: Let $(G_j)_{j\in \mathbb{Z}}$ be a family of disjoint copies of $G$, $X$ a set of size $a$ disjoint from all $G_i$ and $Y=\{y_1,...,y_n\}$ a set disjoint from all $G_i$ and from $X$. For all suitable $i$ and $j$ identify the last vertex of $P_i$ in $G_j$ with the first vertex of $P_i$ in $G_{j+1}$ and add all edges between the copy of $V$ in $G_j$ and $X$.

Let the topology of $G'-Y$ be the simplex topology of $G'_Y$. Define sets $A(i,j,l,\epsilon)$ for $1\leq i\leq n$, $j\leq l\in \mathbb{Z}$ and $1>\epsilon >0$ as follows: Let $A(i,j,l,\epsilon)\cap V$ consist of the vertices of copies of $P_i$ in graphs $G_m$ where $m<j$ or $m>l$. Add the inner points of edges with both end vertices in $A(i,j,l,\epsilon)$. For edges $e$ of which only the starting vertex is in $A(i,j,l,\epsilon)$ add $\iota_e([0,\epsilon\mathclose[)$ and for edges $e$ of which only the end vertex is contained in $A(i,j,l,\epsilon)$ add $\iota_e(\mathopen]1-\epsilon,1])$. The open sets of $G-Y$ and the sets of the form $A(i,j,l,\epsilon)$ form a base of the topology of~$G'$.

\begin{clm} $G'$ is a graph-like space.
\end{clm}
\begin{proof}
For its being a graph-like space we have to show that for each two distinct vertices $v,w$ in $V(G')$ there are two open sets of $G'$ partitioning $V(G')$ such that $v$ and $w$ are not contained in the same open set. If one of these vertices, $v$ say, is not contained in $Y$, then there are two disjoint open sets $U,V$ of $G'-Y$ such that $V\cap V(G')=\{v\}$ and $U\cap V(G')=V(G')\backslash Y-v$. Also for each $i$ there is some $A(i,j,l,\epsilon)$ which does not contain $v$, let $W$ be the union of $U$ and such a $A(i,j,l,\epsilon)$ for all $i$.

So consider the case that both $v$ and $w$ are contained in $Y$ and that $v=y_i$ for some $i$. Pick some suitable $j$ and $l$ and define $U$ to be the union of all $A(m,j,l,\frac{1}{2})$ where $m\neq i$. Let $V$ be the union of $A(i,j,l,\frac{1}{2})$ and all $v^{0.5}$ where $v\in V(G'-U)$. Then $U$ and $V$ are disjoint open sets, one of them contains $v$ and the other $w$ and every vertex is contained in one of these sets.
\end{proof}

In order to show that $G'$ is $k$-connected, note that for each $P_i$ the copies of it induce a pseudo-cycle containing $y_i$. Then $G'$ clearly has a $k$-pseudoflower with $(k-|X|)/2=n$ whose petal sets are of the form $V(G_i)\cup X$.
\end{ex}

\bibliographystyle{plain}
\bibliography{flowers}
\newpage

\section{Appendix}
The goal of this appendix is to show that cycle completions exist, are unique and have the properties we outlined in \cref{sec:cycord}. While there is a lot of prior work in this area, some of which is referenced here, none of it matches our requirements exactly, so we include these proofs for completeness.

This section starts with a short collection of basics about cyclic orders and their connection to linear orders.
In this section we often consider distinct linear or cyclic orders on the same ground set.
Because of this, and in contrast to the rest of the paper, in this section cyclic orders and linear orders are not implicit but introduced more formally as relations on the ground set.
So a cyclic order of a set $S$ is a set $Z\subseteq S\times S\times S$ that is cyclic, antisymmetric, linear and transitive.

The notation of $\leq$ is kept for a linear order, but the linear order in question is added as an index where necessary, for example in $s\leq_L t$ for a linear order $L$.
Similarly, for intervals of cyclic orders, the cyclic order in question may be indicated by an index.

\begin{defn}\cite[Definition 1.1]{Novak84}
Given two linear orders $A$ and $B$ on disjoint ground sets, the linear order $A\oplus B$ is the linear order defined on the union of the ground sets of $A$ and $B$ by letting $x\leq y$ if $x\leq_Ay$ or $x\leq_By$ or $x\in A$ and $y\in B$.
\end{defn}

\begin{defn}\cite[Lemma 1.11, Definition 2.1, Theorem 2.3]{Novak84}
	Given a linear order $L$ on a set $S$, the cyclic order $Z$ \emph{induced by $L$} consists of those triples $(s,s',t)$ of elements of $S$ such that in $L$ one of the equations $s<s'<t$, $s'<t<s$ and $t<s<s'$ holds.
	Given a cyclic order $Z$ on a set $S$, a \emph{cut of $Z$} is a linear order $L$ on $S$ such that $Z$ is the cyclic order induced by $L$.
\end{defn}

\begin{lem}\label{cycsettocut}\cite[Theorem 3.1]{Novak82}
	For every cyclic order $Z$ on set $S$ and every $s$ in $S$ there is a cut of $Z$ whose smallest element is $s$.
\end{lem}

\begin{rem}
	By \cref{cycsettocut}, every non-trivial interval of a cyclic order $Z$ is also an interval of a cut $L$ of $Z$.
	Also, such an interval inherits a linear order from every cut of which it is an interval, and that linear order does not depend on the chosen cut.
\end{rem}

\begin{obs}
	Let $Z$ be a cyclic order on a set $S$, $L$ a cut of $Z$ and $s$, $s'$ and $t$ elements of $S$ such that $(s,s',t)\in Z$.
	If $s\leq t$ in $L$, then $s<s'<t$ in $L$.\qed
\end{obs}

\begin{defn}
    Given a cyclic order $Z$ on ground set $S$, the cyclic order
    \begin{displaymath}
        \{(t,s,r)\colon (r,s,t)\in Z\}
    \end{displaymath}
    on $S$ is the \emph{mirror} of $Z$.
\end{defn}

Monotone maps have the property that preimages of intervals are again intervals.
Maps with that property are close to being monotone:

\begin{lem}\label{homintervaltocyclicorder}
	Let $Z$ be a cyclic order on a set $S$ and $Z'$ a cyclic order on a set $S'$.
	Let $f:S\rightarrow S'$ be a map such that for all intervals $I$ of $Z'$ the set $f^{-1}(I)$ is an interval of $Z$.
	Then $f$ is a monotone map or a mirror of a monotone map.
\end{lem}
\begin{proof}
	If for all elements $r$, $s$ and $t$ of $S$ the implication
	\begin{displaymath}
	(f(r),f(s),f(t))\in Z'\Rightarrow (t,s,r)\in Z
	\end{displaymath}
	holds, then $f$ is a mirror of a monotone map.
	So assume that there are elements $r$, $s$ and $t$ of $S$ such that $(f(r),f(s),f(t))\in Z'$ and $(r,s,t)\in Z$.
	We will start with an observation that we will refer to later in this proof. For this we consider any $u\in S$ such that $(f(r),f(u),f(t))\in Z'$.
	Then $f^{-1}([f(t),f(r)])$ is an interval of $S$ which contains $t$ and $r$ but not $s$, so $[t,r]$ is a subset of $f^{-1}([f(t),f(r)])$.
	As also $u\notin f^{-1}([f(t),f(r)])$, $u\notin [t,r]$ and thus $(r,u,t)\in Z$.
	
	Now let $r'$, $s'$ and $t'$ be elements of $S$ such that $(f(r'),f(s'),f(t'))\in Z'$.
	In order to show that $(r',s',t')\in Z$, first consider the case that the number $n$ of elements in $\{f(r'),f(s'),f(t')\}$ which are not contained in $\{f(r),f(s),f(t)\}$ is zero.
	Assume, by renaming if necessary, that $f(r')=f(r)$, $f(s')=f(s)$ and $f(t')=f(t)$.
	Then by three applications of our observation, $(r,s',t)\in Z$ and thus $(s',t',r)\in Z$ and hence $(r',s',t')\in Z$.
	
	Next consider the case that $n=1$.
	Assume, again by renaming if necessary, that $(f(r),f(r'),f(s))\in Z'$ (see also the left cyclic order of \cref{fig:homintervaltocyclicorder}).
	By our observation, $(r',s,t)\in Z$ and $(r',t,r)\in Z$ and hence also $(r',s,r)\in Z$.
	Then there are three cases:
	Either $f(s')=f(s)$ and $f(t')=f(t)$ or $f(s')=f(t)$ and $f(t')=f(r)$ or $f(s')=f(s)$ and $f(t')=f(r)$.
	In all three cases, by the case $n=0$ also $(r',s',t')\in Z$.
	
	Next consider the case that $n=2$, and that one of the intervals $\mathopen]f(r),f(s)\mathclose[$, $\mathopen]f(s),f(t)\mathclose[$, and $\mathopen]f(t),f(r)\mathclose[$ contains both elements of $\{f(r'),f(s'),f(t')\}$ which are not contained in $\{f(r),f(s),f(t)\}$.
	Assume, by renaming if necessary, that both $f(r')$ and $f(s')$ are contained in $\mathopen]f(r),f(s)\mathclose[$.
	In that case, the fact that $(f(r'),f(s'),f(t'))\in Z'$ implies that $f(r')\in \mathopen]f(r),f(s')\mathclose[$ (see also the middle cyclic order in \cref{fig:homintervaltocyclicorder}).
	Also, by the case $n=1$, $(r',s,t)$, $(s',s,t)$ and $(t,r,r')$ are all contained in $Z$.
	As $f^{-1}([f(t),f(r')])$ contains $t$ and $r'$ but neither $s$ or $s'$, $(r',s,t)\in Z$ implies $[t,r']\subseteq f^{-1}([f(t),f(r')])$.
	Thus $s'$ is not contained in $[t,r']$ and hence $(r',s',t)\in Z$.
	Because $(s',s,t)$ and $(t,r,r')$ are contained in $Z$, also $(r',s',s)$ and $(r',s',r)$ are contained in $Z$.
	By the case $n=0$, also $(r',s',t')\in Z$.
	
	Next consider the case that $n=2$ and none of the intervals $\mathopen]f(r),f(s)\mathclose[$, $\mathopen]f(s),f(t)\mathclose[$ or $\mathopen]f(t),f(s)\mathclose[$ contains two elements of $\{f(r'),f(s'),f(t')\}$.
	Assume, by renaming if necessary, that $f(t')\in \{f(r),f(s),f(t)\}$ and that there is $u\in \{r,s,t\}$ such that $(f(r'),f(u),f(s'))\in Z'$ (see also the right cyclic order of \cref{fig:homintervaltocyclicorder}).
	In this case $(t',r',u)$ and $(u,s',t')$ are both contained in $Z'$ by the case $n=1$ and thus $(r',s',t')$ is also contained in $Z'$.
	
	The only case left is the case $n=3$.
	Assume, by renaming if necessary, that $(f(r'),f(s),f(s'))\in Z'$.
	Then $(s,s',t')$ and $(s,t',r')$ are contained in $Z$ by the case $n=2$ and thus $(r',s',t')\in Z$.
\end{proof}

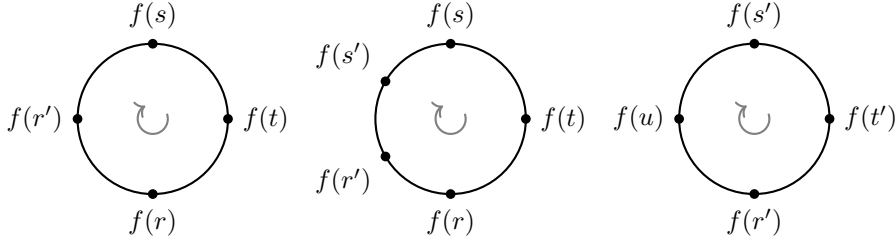
\begin{figure}
	\renewcommand{\flowerradius}{1cm}
	\begin{tikzpicture}
		\basicflower
		\orientation
		\newnode{$f(r)$}{-90}
		\newnode{$f(r')$}{180}
		\newnode{$f(s)$}{90}
		\newnode{$f(t)$}{0}
	\end{tikzpicture}
	\begin{tikzpicture}
		\basicflower
		\orientation
		\newnode{$f(r)$}{-90}
		\newnode{$f(s)$}{90}
		\newnode{$f(t)$}{0}
		\newnode{$f(r')$}{-150}
		\newnode{$f(s')$}{150}
	\end{tikzpicture}
	\begin{tikzpicture}
		\basicflower
		\orientation
		\newnode{$f(r')$}{-90}
		\newnode{$f(s')$}{90}
		\newnode{$f(t')$}{0}
		\newnode{$f(u)$}{180}
	\end{tikzpicture}
	\caption{Three of the cases in the proof of \cref{homintervaltocyclicorder}}\label{fig:homintervaltocyclicorder}
\end{figure}

\begin{defn}\cite[Remark 2.3]{Novak82}
	Given a cyclic order $Z$ of a set $S$ and a subset $S'$ of $S$, the set of triples in $Z$ which only contain elements of $S'$ is a cyclic order on $S'$, the \emph{induced cyclic order} on $S'$.
\end{defn}

\begin{thm}\label{cutsarerelated}\cite[Theorem 3.6]{Novak84}
	Let $Z$ be a cyclic order on set $S$ and let $K$ and $L$ be distinct cuts of $Z$.
	Then there are non-empty disjoint subsets $A$ and $B$ of $S$ such that $A\cup B=S$, $K\restcyc A=L\restcyc A$, $K\restcyc B=L\restcyc B$, $K=K\restcyc A\oplus K\restcyc B$ and $L=K\restcyc B\oplus K\restcyc A$.
\end{thm}

One example of a construction similar to the cycle completion is the following: The cycle completion of a cyclically ordered set $I$ can be obtained from the Dedekind completion of one of its cuts $L$ by adding as many elements to the ground set as necessary such that every element of the original ground set has a predecessor and a successor not in the ground set and then identifying the new smallest and biggest elements.
Also constructing the pseudo-line as in \cite{BCC:graphic_matroids} from $L$, contracting all inner points of an edge to one point and then again identifying the new smallest and biggest elements yields the cycle completion.
Third, the restriction of the cycle completion to the set of cuts is already described in \cite{Novak84}.

We will now give a precise construction which gives a linear order $D(L)$ starting from a linear order $L$ and is very similar to both the Dedekind completion of $L$ and the pseudo-line $L(L)$ as in \cite[Definition 4.1]{BCC:graphic_matroids}.
Similarly to the Dedekind completion, $D(L)$ consists of initial segments of $L$ and of the elements of $L$ itself.
But here an element $l$ of $L$ is not identified with an initial segment of $L$.
The construction of $D(L)$ can be obtained from the pseudo-line $L(L)$ by replacing all the intervals $(0,1)\times \{l\}$ by just $l$.
The topology of the pseudo-line is not needed in the context of this paper.

\begin{ex}\label{ex:cycle completion}
	(See also \cite{BCC:graphic_matroids})
	Let $L$ be a linear order on a set $S$ and let $V(L)$ be the set of \emph{initial segments of $L$}, i.e.\ subsets $S'$ of $S$ which satisfy that if $s$ is an element of $S'$ and $t$ is an element of $S$ with $t<s$ then also $t\in S'$.
	The subset relation is a natural linear order on $V(L)$.
	Define a linear order on the disjoint union of $S$ and $V(L)$ by letting $x\leq y$ if either both $x$ and $y$ are contained in $S$ and $x\leq y$ in $L$ or both are contained in $V(L)$ and $x\leq y$ in $V(L)$ or $x\in y$ or $y\in S\setminus x$.
	Denote the resulting linear order on $S\cup V(L)$ by $D(L)$.
	The smallest element of $D(L)$ is the empty set and the biggest element of $D(L)$ is $S$.
	Denote $S\cup (V(L)\setminus \{S\})$ by $V'(L)$, the restriction of $D(L)$ to $V'(L)$ by $D'(L)$ and the cyclic order induced by $D'(L)$ by $Z(L)$.
	For every element $s$ of $S$, the set $\{t\in S\colon t< s\}$ is the predecessor and the set $\{t\in S\colon t\leq s\}$ is the successor of $s$ in $D(L)$.
	
	Every subset of $D(L)$ has a supremum and an infimum in $D(L)$, which can be seen as follows:
	Given a subset $V'$ of $V(L)$, the set $\bigcup V'$ is an initial segment of $S$ and is the supremum of $V'$ both in $D(L)\restcyc V(L)$ and in $D(L)$.
	Similarly the set $\bigcap V'$ is the infimum of $V'$ in $D(L)\restcyc V(L)$ and $D(L)$.
	So in order to show that every subset of $S\cup V(L)$ has a supremum and an infimum in $D(L)$, it suffices to consider subsets $S'$ of $S$, and by symmetry it suffices to show that $S'$ has a supremum in $D(L)$.
	The set $\{s\in S|\exists \,t\in S':s\leq t\}$, denoted by $S''$, is an initial segment of $S$ which is an upper bound of $S'$.
	Also, no proper subset of $S''$ is an upper bound of $S'$.
	So if $S'$ has an upper bound in $D(L)$ which is less than $S''$, then that upper bound is contained in $S$.
	In particular, as $D(L)$ contains between any two elements of $S$ at least one element of $V(L)$, there is at most one upper bound of $S'$ which is less than $S''$.
	Thus $S'$ has a supremum in $D(L)$.
\end{ex}

\begin{lem}\label{uniquecyccompiso}
	Let $L$ and $K$ be cuts of a cyclic order $Z$ on a set $S$ with at least two elements such that $K=(L\restcyc (S\setminus S'))\oplus (L\restcyc S')$ for an initial segment $S'$ of $L$.
	Then the map
	\begin{displaymath}
	V'(L)\rightarrow V'(K), \quad x \mapsto
	\begin{cases}
	x & x\in S\\
	x\cup (S\setminus S') & x\in V'(L)\setminus S,\ x\subsetneq S'\\
	x \setminus S' & x \in V'(L)\setminus S,\ S'\subseteq x
	\end{cases}
	\end{displaymath}
	is the unique isomorphism of $Z(L)$ and $Z(K)$ which preserves~$S$.
\end{lem}
\begin{proof}
	The map $F_1:V'(L)\rightarrow V'(L\restcyc S')\cup V'(L\restcyc (S\setminus S'))$ which maps elements of $S$ to themselves, initial segments which are properly contained in $S'$ to themselves and initial segments $I$ containing $S'$ to $I\setminus S'$ is an isomorphism of the linear orders $D'(L)$ and $D'(L\restcyc S')\oplus D'(L\restcyc (S\setminus S'))$.
	Similarly the map $F_2:V'(K)\rightarrow V'(L\restcyc S')\cup V'(L\restcyc (S\setminus S'))$ which maps every elements of $S$ to themselves, initial segments properly contained in $S\setminus S'$ to themselves and initial segments $I$ containing $S\setminus S'$ to $I\cap S'$ is an isomorphism of the linear orders $D'(K)$ and $D'(L\restcyc (S\setminus S'))\oplus (L\restcyc S')$.
	Thus the map given in the lemma, which equals $F_2^{-1}\circ F_1$, is an isomorphism of $Z(L)$ and $Z(K)$.
	
	Let $F$ and $G$ be two isomorphisms of $Z(L)$ and $Z(K)$ which preserve $S$.
	Assume for a contradiction that there is $v\in V'(L)$ such that $F(v)$ is less than $G(v)$ in $D'(K)$.
	As $F$ and $G$ both are bijective and preserve $S$, $F(v)$ and $G(v)$ are both contained in $V'(K)\setminus S$.
	Thus there are elements $s\in G(v)\setminus F(v)$ and $t\in S\setminus G(v)$.
	Then $(F(t),F(v),F(s))$ equals $(t,F(v),s)$ and is thus contained in $Z(K)$.
	Because $F$ is monotone, this implies that $(t,v,s)\in Z(L)$.
	But similarly $(s,G(v),t)\in Z(K)$ and thus $(s,v,t)\in Z(L)$, a contradiction.
\end{proof}
\begin{cor}
	Let $Z$ be a cyclic order on set $S$ and let $L$ and $K$ be cuts of $Z$.
	Then there is a unique isomorphism of $Z(L)$ and $Z(K)$ which preserves $S$.
\end{cor}
\begin{proof}
	Let $S'\subseteq S$ such that $K=(L\restcyc (S\setminus S'))\oplus (L\restcyc S')$ and such that $S'$ is an initial segment of $L$.
	Such a set exists by \cref{cutsarerelated}.
	Then the statement follows from \cref{uniquecyccompiso}.
\end{proof}

\begin{lem}\label{uniquecyccomp}
	Let $Z$ be a cyclic order on a non-empty set $S$ and let $\mathcal{V}$ be the set of cuts of $Z$.
	For every cut $L\in \mathcal{V}$ denote the map $V'(L)\rightarrow S\cup \mathcal{V}$ which maps every element of $S$ to itself and every initial segment $S'$ to $(L\restcyc (S\setminus S'))\oplus (L\restcyc S')$ by $\eta_L$.
	Also denote $\{(\eta_L(a),\eta_L(b),\eta_L(c))\colon (a,b,c)\in Z(L)\}$ by $T_L$.
	Then $T_L$ is a cyclic order on $S\cup \mathcal{V}$ which does not depend on the choice of $L$.
\end{lem}
\begin{proof}
	By \cref{cutsarerelated} the maps $\eta_L$ are surjective, so they are bijections between $V'(L)$ and $S\cup \mathcal{V}$.
	Thus every $T_L$ arises from $Z(L)$ by renaming the elements of $V'(L)$ and thus is a cyclic order on $S\cup \mathcal{V}$, and $\eta_L$ is an isomorphism of $Z(L)$ and $T_L$.
	Let $L$ and $K$ be elements of $\mathcal{V}$ and let $S'$ be an initial segment of $L$ such that $K=(L\restcyc (S\setminus S'))\oplus (L\restcyc S')$ (such a segment exists by \cref{cutsarerelated}).
	Denote the unique isomorphism of $Z(L)$ and $Z(K)$ preserving $S$, which exists by \cref{uniquecyccompiso}, by $F$.
	Then $\eta_K\circ F(s)=\eta_L(s)$ for all $s\in S$.
	Also, for all initial segments $I$ of $L$ which are properly contained in $S'$,
	\begin{align*}
	\eta_K\circ F(I)&=\eta_K(I\cup (S\setminus S'))=(K\restcyc (S'\setminus I))\oplus (K\restcyc (I\cup (S\setminus S')))\\
	&=(L\restcyc (S'\setminus I))\oplus (L\restcyc (S\setminus S'))\oplus (L\restcyc I)\\
	&=(L\restcyc (S\setminus I))\oplus (L\restcyc I)=\eta_L(I),
	\end{align*}
	and similarly for all initial segments $I$ of $L$ which contain $S'$
	\begin{align*}
	\eta_K\circ F(I)&=\eta_K(I\setminus S')=(K\restcyc (S\setminus (I\setminus S')))\oplus (K\restcyc (I\setminus S'))\\
	&=(L\restcyc (S\setminus I)))\oplus (L\restcyc S')\oplus (L\restcyc (I\setminus S'))\\
	&=(L\restcyc (S\setminus I)) \oplus (L\restcyc I)=\eta_L(I).
	\end{align*}
	So $\eta_K\circ F=\eta_L$ and thus $\eta_K\circ F\circ \eta_L^{-1}$ is the identity.
	But $\eta_K\circ F\circ \eta_L^{-1}$ is also a composition of isomorphisms of cyclic orders and thus the identity is an isomorphism of $T_L$ and $T_K$, so $T_L=T_K$.
\end{proof}

Given a cyclic order $Z$ on set $S$ and a cut $L$ of $Z$, the previous lemma shows that $T_L$ only depends on $Z$ and not on $L$.
$T_L$ will be a cycle completion of $Z$.
From now on, denote $T_L$ by $\cyccomp(Z)$ and its ground set by $\cycset(Z)$.

The next lemma shows that $\cyccomp(Z)$ is really a cycle completion of a cyclic order $Z$.
As a result, cycle completions of cyclic orders exists.

\begin{lem}\label{boundsofintervalincompletion}
	Let $Z$ be a cyclic order on set $S$.
	Then for every non-trivial interval $I$ of $S$ there are unique elements $v$ and $w$ of $\cycset(Z)\setminus S$ such that $I=[v,w]\cap S$.
\end{lem}
\begin{proof}
	Let $L$ be a cut of $Z$ such that $I$ is an interval of $L$ and such that some element of $S$ is bigger than all elements of $I$ in $L$.
	By construction of $D(L)$ there are unique elements $v$ and $w$ of $V'(L)\setminus S$ such that $I=[v,w]\cap S$ in $D(L)$.
	Then $v$ and $w$ are also the unique elements of $V'(L)\setminus S$ such that $I=[v,w]\cap S$ in $Z(L)$, and thus $\eta_L(v)$ and $\eta_L(w)$ are the unique elements of $\cycset(Z)$ such that $I=[\eta_L(v),\eta_L(w)]\cap S$ in $\cyccomp(Z)$.
\end{proof}

Given a cyclic order $Z$ on set $S$, $\cyccomp(Z)$ clearly has the property that for distinct elements $v$ and $w$ of $\cycset(Z)\setminus S$ the interval $[v,w]\cap S$ is non-trivial.
This property holds for cycle completions in general, as the following rephrasing of \cref{distinctvertices} shows:

\begin{lem}\label{distinctcutsboundinterval}
    Let $Z$ be a cyclic order on a set $S$ with at least two elements and let $T$ on the set $R$ be a cycle completion of $Z$.
    Then for distinct $v$ and $w$ in $R\setminus S$ the interval $[v,w]\cap S$ of $S$ is non-trivial.
\end{lem}
\begin{proof}
    As $S$ has at least two elements, it has a non-trivial interval and thus there are $x$ and $y$ in $R\setminus S$ such that $[x,y]\cap S$ is a non-trivial interval of $S$.
    
    If $[v,y]\cap S$ is a trivial interval, then one of $[v,y]\cap S$ and $[y,v]\cap S$ is empty. Suppose $v \not= w$.
    If $[v,y]\cap S$ is empty, then $x\notin [v,y]$, implying that $v\in [x,y]$ and $[x,v]\cap S=[x,y]\cap S$, a contradiction.
    Similarly, if $[y,v]\cap S$ is empty, then also $[x,v]\cap S=[x,y]\cap S$, a contradiction to $T$ being a cycle completion.
    Hence if $[v,y]\cap S$ is a trivial interval of $S$ then $v=y$.

    Thus either $[v,y]\cap S$ is a non-trivial interval of $S$, or $v=y$ which in particular implies that $[v,x]\cap S$ is a non-trivial interval of $S$.
    As $v\neq w$, similarly $[v,w]\cap S$ is a non-trivial interval of $S$.
\end{proof}

In this paper, cycle completions are used as index sets for $k$-pseudoflowers, and they are related to each other via surjective monotone maps that map only cuts to cuts.
The following lemmas establish a few basic facts about such monotone maps.
In particular the following lemma is a rephrasing of \cref{cyclichoms2}.

\begin{lem}\label{fromcyccomphomtoalternatives}
	Let $Z$ and $Z'$ be cyclic orders on sets $S$ and $S'$ with at least two elements.
	Let $T$ and $T'$ be cycle completions of $Z$ and $Z'$ on sets $R$ and $R'$.
	Let $F:R \rightarrow R'$ be surjective and monotone such that $F(S)\subseteq S'$.
	Then there is for every $v'\in R'\setminus S'$ exactly one $v\in R$ with $F(v)=v'$, and there is for every $s'\in S'$ some $s\in S$ with $F(s)=s'$.
\end{lem}
\begin{proof}
	Every interval of $T$ with at least two elements contains at least one element of $S$.
	As $F^{-1}(v')$ does not contain elements of $S$, it has at most one element.
	Because $F$ is surjective, $F^{-1}(v')$ contains exactly one element $v$.
	
	Let $t'$ and $q'$ be the predecessor and successor of $s'$ in $T'$.
	Then $[t',q']$ taken in $T'$ consists of $t'$, $q'$ and $s'$.
	Let $t$ and $q$ be the unique elements of $R$ such that $F(t)=t'$ and $F(q)=q'$, and let $s$ be an element of $[t,q]$ taken in $T$ which is contained in $S$.
	As $F$ is monotone and $F(s)$, $q'$ and $t'$ are pairwise disjoint, $F(s)\in [t',q']$ in $T'$.
	Thus $F(s)=s'$.
\end{proof}

So $F$ naturally induces two other monotone maps: The restriction of $F$ to $S$ is surjective and monotone, and $g:R'\setminus S'\rightarrow R\setminus S$ which maps every element to its unique preimage under $F$ is injective and monotone.
In the other direction, all surjective monotone $f$ from $Z$ to $Z'$ are derived from a surjective monotone map $T\rightarrow T'$ with $F(S)\subseteq S'$.

\begin{lem}\label{preimageofintervalbounds}
	Let $Z$ and $Z'$ be cyclic orders on set $S$ and $S'$ respectively, each with at least two elements, and let $T$ and $T'$ be cycle completions of $Z$ and $Z'$ on sets $R$ and $R'$ respectively.
	Let $F:R\rightarrow R'$ be surjective and monotone with $F(S)\subseteq S'$.
	Then for all elements $v$ and $w$ of $R$ such that $F(v)$ and $F(w)$ are distinct elements of $R'\setminus S'$ the equations $F^{-1}(\mathopen]F(v),F(w)\mathclose[)=\mathopen]v,w\mathclose[$ and $F^{-1}([F(v),F(w)])=[v,w]$ hold.
\end{lem}
\begin{proof}
	By \cref{fromcyccomphomtoalternatives}, $v$ is the only element of $R$ which is mapped to $F(v)$ by $F$ and similarly for $w$.
	So for all $x\in R-v-w$, $(v,x,w)\in T$ if and only if $(F(v),F(x),F(w))\in T'$ and thus the two equations hold.
\end{proof}

Together with \cref{fromcyccomphomtoalternatives}, the previous lemma shows \cref{cyclichoms3}.
With its help, it is now possible to show the following phrasing of \cref{cyclichoms1}.

\begin{lem}\label{fromThomtoTcychom}
	Let $Z$ and $Z'$ be cyclic orders on sets $S$ and $S'$ and let $T$ and $T'$ be cycle completions of $Z$ and $Z'$ respectively on sets $R$ and $R'$.
	Let $f:S\rightarrow S'$ be surjective and monotone.
	If $S'$ has at least two elements, then there is a unique surjective monotone map $F$ from $T$ to $T'$ such that the restriction of $F$ to $S$ equals $f$.
\end{lem}
\begin{proof}
    For $v\in R\setminus S$, define $F(v)$ as follows:
    If there are $s$ and $t$ in $S$ such that $v\in [s,t]$ and $[s,t]\cap S \subseteq f^{-1}(s')$ for some $s'\in S'$, then let $F(v)=s'$.
    This is well defined as there is at most one such $s'$.
    Otherwise let $s$ and $t$ be elements of $S$ with $f(s)\neq f(t)$ and $v\in [s,t]$.
    Then $S_1:=f([s,v]\cap S)$ and $S_2:=f([v,t]\cap S)$ are intervals of $S'$, and they are non-trivial because they are disjoint.
    So they can be written uniquely as $S_1=[w_1,w_1']\cap S'$ and $S_2=[w_2,w_2']\cap S'$ for elements $w_1,w_1',w_2,w_2'$ of $R'\setminus S'$.
    Then $w_1'=w_2$, and this element of $R'$ does not depend on the choice of $s$ and $t$.
    Let $F(v)=w_1'$.
    In particular $F(v)\in R'\setminus S'$.

    In order to show that $F$ is surjective, consider $v'\in R'\setminus S'$.
    Let $s'$ and $t'$ be distinct elements of $S'$ with $v'\in [s',t']$ and let $s$ and $t$ be elements of $S$ with $f(s)=s'$ and $f(t)=t'$.
    Then $f^{-1}([s',v']\cap S')$ is a non-trivial interval of $S$ and thus can be written uniquely as $[v,w]\cap S$ for elements $v$ and $w$ of $R\setminus S$.
    As $[s,v]\cap S \subseteq f^{-1}([s',v']\cap S')$ and $[v,t]\cap S\subseteq f^{-1}([v',t']\cap S')$, this implies that $F(v)=v'$.
    
    In order to show that $F$ is monotone, consider $s,t,v\in R$ such that the triple $(F(s),F(v),F(t))$ is contained in~$ T'$.
    For an element $r$ of $S'$, $F^{-1}(r)\setminus S$ only contains elements of $R\setminus S$ whose image under $F$ is defined via the first case, which implies that $F^{-1}(r)$ is an interval of $R$.
    So in the case where all three elements $F(s)$, $F(t)$ and $F(v)$ are contained in $S'$, $(s,v,t)\in T$ follows from the fact that $f$ is monotone.
    As between any two elements of $R'\setminus S'$ there is some element of $S'$, in order to show $(s,v,t)\in T$ it suffices to consider the case that $F(s)$ and $F(t)$ are contained in $S'$ but $F(v)$ is not.
    Again, because $F^{-1}(r)$ is an interval of $R$ for every $r\in S'$, it suffices to consider the case that both $s$ and $t$ are contained in $S$.
    The fact that $F(v)$ is not contained in $S'$ implies that $v$ is not contained in $S$ and that $F(v)$ is defined via the second case.
    In particular if $v\in [t,s]$ then $F(v)\in [f(t),f(s)]=[F(t),F(s)]$, which is impossible.
    Hence $(s,v,t)\in T$.
\end{proof}

By the previous lemma, given two cycle completions of a cyclic order $Z$ on set $S$, the identity on $S$ can be extended uniquely to a surjective monotone map between the cycle completions, showing that cycle completions are essentially unique.
Together with the existence of cycle completions proved above this proves \cref{existencecyccomp}.

\end{document}